\documentclass[11pt, a4paper]{amsart}

\usepackage{amsmath}
\usepackage{tikz}
\usepackage{amssymb}

\usepackage{amsthm}
\usepackage{tikz-cd}
\usepackage[all]{xy}
\usepackage{amsfonts}
\usepackage{mathrsfs}
\usepackage{hyperref}
\usepackage{xcolor}
\usepackage{MnSymbol} 

\usepackage{mathtools}
\usepackage{subfig, caption} 
\usepackage{wrapfig}
\input{header_labeling.tex}

\newcommand{\CC}{\mathbb{C}}

\newcommand{\HH}{\mathbb{H}}

\newcommand{\NN}{\mathbb{N}}

\newcommand{\QQ}{\mathbb{Q}}
\newcommand{\RR}{\mathbb{R}}

\newcommand{\ZZ}{\mathbb{Z}}

\newcommand{\cA}{\mathcal{A}}

\newcommand{\cC}{\mathcal{C}}

\newcommand{\cF}{\mathcal{F}}

\newcommand{\cH}{\mathcal{H}}

\newcommand{\cM}{\mathcal{M}}

\newcommand{\cO}{\mathcal{O}}

\newcommand{\cT}{\mathcal{T}}

\newcommand{\sC}{\mathscr{C}}
\newcommand{\sD}{\mathscr{D}}
\newcommand{\sE}{\mathscr{E}}
\newcommand{\sF}{\mathscr{F}}

\newcommand{\fA}{\mathsf{A}}
\newcommand{\fB}{\mathsf{B}}
\newcommand{\fC}{\mathsf{C}}

\newcommand{\fE}{\mathsf{E}}

\newcommand{\fG}{\mathsf{G}}

\newcommand{\fW}{\mathsf{W}}

\newcommand{\fkH}{\mathfrak{H}}
\newcommand{\fkN}{\mathfrak{N}}

\newcommand{\RP}{\operatorname{\RR P}}
\newcommand{\CP}{\operatorname{\CC P}}

\newcommand{\Homeo}{\operatorname{Homeo}}
\newcommand{\Homeop}{\operatorname{Homeo^+}}
\newcommand{\Isom}{\operatorname{Isom}}

\newcommand{\closure}[1]{\overline{#1}}

\newcommand{\Fix}[1]{\operatorname{Fix}(#1)}

\newcommand{\SL}{\operatorname{SL}_2}
\newcommand{\GL}{\operatorname{GL}_2}
\newcommand{\PSL}[1]{\operatorname{PSL}_2(#1)}
\newcommand{\PGL}[1]{\operatorname{PGL}_2(#1)}

\newcommand{\Aut}{\operatorname{Aut}}
\newcommand{\Inn}{\operatorname{Inn}}
\newcommand{\Out}{\operatorname{Out}}
\newcommand{\Mod}{\operatorname{Mod}}

\newcommand{\id}{\operatorname{Id}}

\newcommand{\im}{\operatorname{Im}}

\newcommand{\tr}{\operatorname{tr}}

\newcommand{\Spec}{\operatorname{Spec}}

\newcommand{\card}{\operatorname{card}}

\title[Length spectra of $\operatorname{Out}(F_2)$ and $\mathcal{H}_2$]{Uniform difference between the length spectra of $\operatorname{Out}(F_2)$ and the genus two handlebody group}

\author[Kim]{KyeongRo Kim}
\address{\hskip-\parindent
Research institute of Mathematics\\
Seoul National University\\
GwanAkRo 1, Gwanak- Gu, Seoul 08826, Korea}
\email{kyeongrokim14@gmail.com}

\author[Seo]{Donggyun Seo}
\address{\hskip-\parindent
Research institute of Mathematics\\
Seoul National University\\
GwanAkRo 1, Gwanak- Gu, Seoul 08826, Korea}
\email{seodonggyun@snu.ac.kr}

\date{\today}

\begin{document}

\begin{abstract} 
In this paper, we analyze the natural homomorphism from the genus $g$ handlebody group to the outer automorphism group of the free group with rank $g$, in terms of length spectra. In general, the preimage of each fully irreducible outer automorphism contains a potentially infinite number of pseudo-Anosov mapping classes. Our study reveals a crucial relationship: for any pseudo-Anosov map in this preimage, its stretch factor must equal or exceed that of the corresponding fully irreducible outer automorphism. Notably, in the case of genus two, we establish that the minimum stretch factor among these pseudo-Anosov maps is less than ten times the stretch factor of the fully irreducible outer automorphism. These results partially address a question by Hensel and have practical implications, including a lower bound for the geodesic counting problem in the genus two handlebody group.
\end{abstract}

\subjclass{Primary ; Secondary }
 
\maketitle

\section{Introduction} 

\subsection{The Main Question}Since Thurston's classification of mapping class elements \cite{Thurston88}, the class of pseudo-Anosov maps has played an important role in the fields of geometric topology and geometric group theory. After Thurston's work, Bestvina and Handel \cite{BestvinaHandel92} studied irreducible elements in the outer automorphism group of a free group, as an analogous concept to pseudo-Anosov maps. This naturally raises the question of how pseudo-Anosov maps are related to irreducible elements. 

One possible candidate for establishing the connection is the handlebody groups. Let $g$ be an integer bigger than $1$. The \emph{genus $g$ handlebody group} $\cH_g$ is the mapping class group of the genus $g$ handlebody $V_g$. The boundary $\partial V_g$ of $V_g$ is the genus $g$ surface.
Naturally, the inclusion from $\partial V_g$ to $V_g$ induces an injective homomorphism from $\cH_g$ to $\Mod(\partial V_g)$ (See \cite[Lemma~3.1]{Hensel20}). In this sense, we think of  $\cH_g$ as a subgroup of $\Mod(\partial V_g)$. 
On the other hand, since  the fundamental group of $V_g$ is the free group $F_g$ of rank $g$, the action of $\cH_g$ on $\pi_1(V_g)$ induces a surjective homomorphism
$A$ from $\cH_g$ to the outer automorphism group  $\Out(F_g)$ of $F_g$. 

One may expect that for any pseudo-Anosov map $\psi$ in $\cH_g$, $A(\psi)$ is irreducible, or that for any irreducible element $\varphi$ in $\Out(F_g)$, every element in $A^{-1}(\varphi)$ is pseudo-Anosov. Unfortunately, these are not true in general.
Nonetheless, for any irreducible element $\varphi$ in $\Out(F_g)$, we may find a pseudo-Anosov element in $A^{-1}(\varphi)$. Therefore, it is still meaningful to understand the relation between $\varphi$ and the pseudo-Anosov elements in $A^{-1}(\varphi)$. 

In this perspective, Hensel \cite{Hensel20} proposed the following question. Oertel also asked a similar question \cite[Problem 9.8]{Oertel02}.
\begin{ques}
[Question 6.8  in [\cite{Hensel20}]\label{Que:Hensel}
Let $\varphi$ be a fully irreducible element in $\Out(F_g)$. Consider the action of $\varphi$ on the Culler--Vogtmann Outer space $\mathrm{CV}_g$. How does the translation length with respect to the Lipschitz metric $d_{\mathrm{CV}}$ relate to possible translation lengths of $\psi\in A^{-1}(\varphi)$ acting on the Teichm\"uller space $\cT(\partial V_g)$ with the Thurston metric $d_{Th}$? 
\end{ques}

The terms ``translation length'' are ambiguous since there are two standard notions for measuring translation length on a metric space. One is the minimal translation length and the other is the stable translation length. Let $(X,d)$ be a metric space and $\gamma$ an isometry on $(X,d)$. The \emph{minimal translation length} $t_{(X,d)}(\gamma)$ of $\gamma$ is defined as 
$$t_{(X,d)}(\gamma)=\inf_{x\in X}d(x,\gamma(x)).$$
The \emph{stable translation length} $\tau_{(X,d)}(\gamma)$ of $\gamma$ is defined as follows. Given $x_0\in X$, we can see that 
$$\tau_{(X,d)}(\gamma)=\lim_{n\to \infty} \frac{d(x_0,\gamma^n(x_0))}{n}$$
exists. Moreover, $\tau_{(X,d)}(\gamma)$ does not depend on the choice of $x_0$.
However, both the Lipschitz metric and the Thurston metric are not actual metrics since those are asymmetric. 
In spite of the lack of symmetry, both minimal and stable translation lengths are well defined for the isometries of $(\cT(\partial V_g),d_{Th})$ and $(\mathrm{CV}_g,d_{\mathrm{CV}})$. 

In fact, the Teichm\"uller space is not uniquely geodesic under the Thurston metric while the Teichm\"uller space is uniquely geodesic under the Teichm\"uller metric $ d_{Teich}$. This fact makes it difficult to understand the minimal translation length of a pseudo-Anosov map with respect to $d_{Th}$. However, for the stable translation lenght of $d_{Th}$,  it is known by \cite[Corollary ~1.3]{PapaSu16} that $$\tau_{(\cT(\partial V_g), d_{Th})}(\gamma)=t_{(\cT(\partial V_g), d_{Teich})}(\gamma)$$ for all pseudo-Anosov maps $\gamma$ in $\Mod(\partial V_g)$. In this reason, we use the Teichm\"uller metric instead of the Thurston metric. Also,  translation length refers to the minimal translation length.

\subsection{The Main Result}

Let $\varphi$ be a fully irreducible element in $\Out(F_g)$, and let $T(\varphi)$ denote the set of translation lengths $t_{(\cT(\partial V_g), d_{Teich})}(\psi)$ of pseudo-Anosov elements $\psi$ in $A^{-1}(\varphi)$. In general, $T(\varphi)$ is an unbounded subset of  the real line $\RR$. However, by the classical result of the length spectrum of a mapping class group (see \refthm{discretSpec} for the statement), we can see that $T(\varphi)$ forms a closed and discrete subset of $\RR$. Thus, it is worthwhile to determine the minimum of $T(\varphi)$.

In this paper, we first show that $t_{(\mathrm{CV}_g,d_{\mathrm{CV}})}(\varphi)$ provides a lower bound of $T(\varphi)$. 
\begin{restate}{Theorem}{Thm:lowerBound}
For every $g>1$ and a pseudo-Anosov mapping class $\psi \in \mathcal{H}_g$, if $A(\psi)$ is fully irreducible, then $$ t_{(\mathrm{CV}_g, d_{\mathrm{CV}})}(A(\psi))\leq t_{(\cT(\partial V_g), d_{Teich})}(\psi).$$
\end{restate}
Especially, for the genus two case, we show that we can always find a pseudo-Anosov element in $A^{-1}(\varphi)$, the translation length of which is slightly bigger than the translation length of $\varphi$.
\begin{restate}{Theorem}{Thm:almostMinimal}
    For each fully irreducible outer automorphism $\varphi \in \Out(F_2)$, there exists a pseudo-Anosov mapping class $\psi \in A^{-1}(\varphi)$ such that $$0 \leq t_{(\cT(\partial V_2), d_{Teich})}(\psi) - t_{(\mathrm{CV}_2, d_{\mathrm{CV}})}(\varphi) < \log 10.$$
\end{restate}
By \refthm{almostMinimal}, we can conclude that the minimum of $T(\varphi)$ lies in the closed interval $[t_{(\mathrm{CV}_2, d_{\mathrm{CV}})}(\varphi),t_{(\mathrm{CV}_2, d_{\mathrm{CV}})}(\varphi) +\log 10]$. In other words, the difference between the minimum of $T(\varphi)$ and $t_{(\mathrm{CV}_2, d_{\mathrm{CV}})}(\varphi)$ is bounded by $\log10$. Moreover, this bound does not depend on $\varphi$. Therefore,
\refthm{almostMinimal} gives a partial answer for the genus two case of \refque{Hensel}.
\subsection{An Application: the Geodesic Counting Problem}

As a quick application of \refthm{almostMinimal}, we consider the geodesic counting problem in the genus two handlebody group.
This is motivated by Eskin--Mirzakhani's work \cite{EskinMirzakhani11}, Chowla--Cowles--Cowles \cite{ChowlaCowlesCowles80}, and Kapovich--Pfaff \cite{KapovichPfaff18}. In \cite{EskinMirzakhani11}, Eskin--Mirzakhani solved the geodesic counting problem in the moduli space of a compact surface. Similarly, Kapovich--Pfaff \cite{KapovichPfaff18} studied the geodesic counting problem for the outer automoprhism group of a free group.

For each $R > 0$, we write $\fkH_2(R)$ for the number of conjugacy classes of pseudo-Anosov mapping classes of translation length at most $R$ in the genus two handlebody group. Due to the following theorem, we can see that $\fkH_2(R)$ is finite for all $R>0$ and so the geodesic counting problem for a handlebody group is well defined.
\begin{restate}{Thorem}{Thm:discretSpecH}
    Let $g \geq 2$. For any $D\geq 1$, there exists only finitely many conjugacy classes of pseudo-Anosov elements of $\cH_g$ with stretch factor at most $D$.
\end{restate}
 Then, the classical estimation for the geodesic counting problem in $\Out(F_2)$ (see \cite{KapovichPfaff18}), combined with \refthm{almostMinimal},
implies the following theorem.

\begin{restate}
{Theorem}{Thm:countH2}
For any $\epsilon>0$, there is $R_\epsilon>0$ such that 
$$(1-\epsilon)\frac{e^{2R}}{400R}< \fkH_2(R) < \infty$$ for all $R> R_\epsilon$.
\end{restate}

\refthm{countH2} says that $\fkH_2(R)$ has at least exponential asymptotic growth.

\subsection{A Historical Remark}

\refthm{almostMinimal} establishes a direct connection between the geometries of the genus two handlebody group and the outer automorphism group of the rank-two free group, as implied by the homomorphism $A$.
While numerous researchers \cite{Zieschang61, McMillan63, Luft78, McCullough85, Hain08} have intuitively constructed and analyzed the homomorphism $A$ in various ways, our approach focuses on investigating its length-spectral and dynamical behavior, which distinguishes it from previous methods.

Geometric simiarity between handlebody groups and outer automorphism groups of free groups has  been discovered indirectly. 
In  \cite{Masur86}, Masur studied an action of a handlebody group on the space of projective measured laminations.
He proved that there exists a unique limit set such that the handlebody group acts properly discontinuously on the complement of its closure.
Notice, on the contrary, an orbit of the surface mapping class group is dense.
In this dynamical point of view, an outer automorphism group of a free group is more similar to a handlebody group than to a mapping class group.
That is to say, Guirardel \cite{Guirardel00} found a unique limit set on the boundary of $\mathrm{CV}_g$ so that $\Out(F_g)$ acts properly discontinuously on the complement of its closure.

The genus two handlebody group is comparatively more accessible than higher genus handlebody groups. This assertion was supported by Hamenst\"adt--Hensel \cite{HamenstadtHensel21}, who demonstrated that the genus two handlebody group acts properly cocompactly on a $\mathrm{CAT}(0)$ cube complex, rendering it biautomatic and possessing a quadratic Dehn function. Furthermore, Chesser \cite{Chesser22} and Chesser--Leininger \cite{ChesserLeininger23} provided comprehensive characterizations of finitely generated stable subgroups and purely pseudo-Anosov subgroups within the genus two handlebody group.

\subsection{Organization}

We review and introduce basic terminologies in \refsec{preliminary}.
In \refsec{spectrum}, we investigate the topology of the length spectrum of a handlebody group,  reviewing basic properties about translation length of Teichm\"uller spaces and Culler--Vogtmann outer spaces. Some statements in this section have not yet been proven before. 
In \refsec{OutF2}, we arrange well-known facts about $\Out(F_2)$. This section does not contain any new result, but we will use these facts without mentioned in the succeeding sections. In \refsec{trace},
We analyze the trace and the anti-trace of an element in $\GL(\mathbb{Z})$.
In \refsec{GenusTwoHandlebody}, we show the main theorem, \refthm{almostMinimal}. In \refsec{GeodesicCounting}, we discuss the geodesic counting problem in a handlebody group as an application of \refthm{almostMinimal} and show \refthm{countH2}.

\subsection{Acknowledgement}

We would like to thank to Hyungryul Harry Baik, Inhyeok Choi, Hongtaek Jung, and  Chandrika Sadanand for helpful conversations.
The first author was supported by the National Research  Foundation of Korea Grant funded by the Korean Government (NRF-2022R1C1C2009782).
The second author was supported by the National Research Foundation of Korea Grant funded by the Korean Government (NRF-2021R1C1C2005938).

\section{Preliminary} \label{Sec:preliminary}
In this section, we review and introduce basic notions. 

First, we write  $\cO_G(h)$ for the conjugacy class of $h$ in a group $G$. In many cases, it is necessary to specify the ambient group $G$. 

Then, $S_{g,n}^b$ denotes the connected orientable surface of genus $g$ with $n$ puntures and $b$ boundary components.
The \emph{mapping class group} $\Mod(S_{g,n}^b)$ of $S_{g,p}^b$ is the group of isotopy classes of orientation presering homeomorphisms of $S_{g,p}^b$, fixing the boundary pointwise. More generally, we also define the \emph{extended mapping class group} $\Mod^\pm(S_{g,n}^b)$ of $S_{g,p}^b$ to be the group of isotopy classes of homeomorphisms of $S_{g,p}^b$, fixing the boundary pointwise. We denote the stretch factor of a pseudo-Anosov map $\varphi$ by $\lambda_\varphi$. It is well known that for a pseudo-Anosov map $\varphi$ in $\Mod(S_{g,n})$,  $\lambda_\varphi=\lambda_\psi$ for all $\psi\in \cO_{\Mod(S_{g,n})}(\varphi)$. Also, we denote the Teichm\"uller space of $S_{g,p}$ by $\cT(S_{g,p})$. It is well known that the (minimal) translation length of a pseudo-Anosov map $\varphi$ with respect to $d_{Teich}$ is equal to $\log \lambda_\varphi$, that is, $$t_{(\cT(S_{g,p}),d_{Teich})}(\varphi)=\log \lambda_\varphi.$$ 

 For convenience, we do not distinguish strictly a curve  with the isotopy class of that and a homeomorphism  with the corresponding mapping class. When  $\alpha$ is a curve on a surface, we denote the \emph{Dehn twist}  about $\alpha$ by $T_\alpha$. In this paper, the Dehn twist refers to the right Dehn twist.

The \emph{genus $g$ handlebody} $V_g$ is  the three manifold with boundary obtained from the 3–ball by attaching $g$ one-handles. A curve $m$ in $\partial V_g$ is called a \emph{meridian} of $V_g$ if there is a properly embedded closed disk $D$ in $V_g$, the boundary of which is $m$. The \emph{genus $g$ handlebody group} $\cH_g$ is the group of isotopy classes of orientation preserving homeomorphisms of $V_g$. As mentioned in  the introduction, it is well known that the inclusion map  $i:\partial V_g  \hookrightarrow V_g$ induces the injective homomorphism $i_\#:\cH_g\hookrightarrow \Mod(\partial V_g)$. See \cite[Section~3]{Hensel20} for the detailed discussion. Hence, we can think of $\cH_g$ as a subgroup of $\Mod(\partial V_g)$. 

We denote the free group of rank $g$ by $F_g$ and the outer automoprhism group of $F_g$ by $\Out(F_g)$. Note that $\pi_1(V_g)\cong F_g$.
Let $g$ be an integer bigger than $1$. After fixing an identification between $\pi_1(V_g)$ and $F_g$, the action of $\cH_g$ on $\pi_1(V_g)$ induces natually the surjective homeomoprhism $A:\cH_g\to \Out(F_g)$. It is proven by Luft \cite{Luft78} that the kernel $\ker(A)$ of $A$ is infinitely generated by the Dehn twists about meridians. See also \cite[Section~6]{Hensel20} for a survey about the properties of $A$.

In this paper, we denote the stretch factor of a fully irreducible element $\varphi$ in $\Out(F_g)$, $g>1$ by $\mu_\varphi$. As in Teichm\"uller spaces, it is known that for each fully irreducible element $\varphi$ in $\Out(F_g)$, 
$$t_{(\mathrm{CV}_g,d_{\mathrm{CV}})}(\varphi)=\log \mu_\varphi.$$ Also, the stretch factor of $\varphi$ can be estimated as follows. Fix a free base $S$ of $F_g$.
For a word $w$ in $F_g$ with respect to $S$,  we denotes the \emph{word length} of $w$ by $|w|_S$. Then, $$|\cO_{\Out(F_g)}(w)|_S:=\min_{v\in\cO_{\Out(F_g)}(w)}|v|_S.$$ As $\varphi$ is fully irreducible, we can take a word $w_0$ in $F_g$ with respect to $S$ so that $\cO_{\Out(F_g)}(\varphi^n(w_0))\neq \cO_{\Out(F_g)}(\varphi^m(w_0))$ whenever $n\neq m$. 
Then, the limit 
$$\lim_{n\to \infty}\sqrt[n]{|\cO_{\Out(F_g)}(\varphi^n(w_0))|_S}$$
exists and it does not depend on the choices of $S$ and $w_0$. Moreover, it is well known that 
$$\mu_\varphi=\lim_{n\to \infty}\sqrt[n]{|\cO_{\Out(F_g)}(\varphi^n(w_0))|_S}.$$ From this, we can  see that the stretch factor is invariant under  conjugation.

For every square matrix $M$, the trace of $M$ is denoted by $\tr(M)$.
The \emph{anti-trace}, the sum of the entries of the anti-diagonal, is written by $\tr^*(M)$.

\section{ Spectra} \label{Sec:spectrum}

\subsection{The length spectrum of a moduli space}

Let $S$ be a hyperbolic surface of finte type. The \emph{Teichmuller length spectrum} $\Spec(\cM(S))$ of moduli space $\cM(S)$  is the set 
$$\{\log(\lambda_\varphi): \text{$\varphi$ is a pseudo-Anosov mapping class in $\Mod(S)$}\}.$$

The following theorem is well known. See \cite[Theorem~14.9]{FarbMargalit12} for the detailed exposition. 

\begin{thm}\label{Thm:discretSpec}
Let $g,n\geq 0$ and $S_{g,n}$ be a hyperbolic surface of finite type with genus $g$ and $n$ punctures. For any $D\geq 1$, there exists only finitely many conjugacy classes of pseudo-Anosov elements of $\Mod(S_{g,n})$ with stretch factor at most $D$. In particular,  $\Spec(\cM(S_{g,n}))$ is a  closed, discrete subset of $\RR$.
\end{thm}

Similarly, we can define the \emph{spectrum} for a Handlebody group $\cH_g$ as follows:
$$\Spec(\cH_g):=\{\log(\lambda_\varphi): \text{$\varphi$ is a pseudo-Anosov mapping class in $\cH_g$}\}.$$

Then, in virtue of the following theorem, we can show that the above theorem can descent to the case for handlebody groups. See \cite[Theorem~5.14]{Hensel20}.

\begin{thm}\label{Thm:finiteConj}
Let $\varphi$ be any pseudo-Anosov mapping class in $\Mod(S_g)$. Then, $\varphi$ can be contained in at most finitely many conjugates of $\cH_g$ in $\Mod(S_g)$.
\end{thm}

\begin{lem}\label{Lem:trivialNorm}
Let $g\geq 2$. The normalizer of $\cH_g$ in $\Mod(S_g)$ is exactly $\cH_g$.
\end{lem}
\begin{proof}
Let $h$ be an element in $\Mod(S_g)$ with $h\cH_g h^{-1}=\cH_g$. For any meridian $\alpha$ of $V_g$, the conjugation of the Dehn twist about $\alpha$ by $h$ is the Dehn twist about $h(\alpha)$, that is, $hT_\alpha h^{-1}=T_{h(\alpha)}$. Since, by \cite[Theorem~5.6]{Hensel20}, every Dehn twist in $\cH_g$ is the Dehn twist about a meridian, we can conclude that $h$ maps each meridian to a meridian. Therefore,  by \cite[Corollary~5.11]{Hensel20}, $h\in \cH_g$. This implies that the normalizer of $\cH_g$ in $\Mod(S_g)$ is a subgroup of $\cH_g$. Thus, the normalizer is exactly $\cH_g$.     
\end{proof}

\begin{thm}\label{Thm:discretSpecH}
Let $g \geq 2$. For any $D\geq 1$, there exists only finitely many conjugacy classes of pseudo-Anosov elements of $\cH_g$ with stretch factor at most $D$. In particular,  $\Spec(\cH_g)$ is a closed, discrete subset of $\RR$.
\end{thm}
\begin{proof}
Let $\varphi$ be a pseudo-Anosov element in $\cH_g$.
By \refthm{finiteConj}, we can take a finite sequence $\{h_1, h_2,\cdots, h_k\}$ of elements of $\Mod(S_g)$ so that  if $\varphi\in h^{-1}\cH_g h$ for some $h\in \Mod(S_g)$, then there is a unique $i\in \{1,2, \cdots, k\}$ such that $h^{-1}\cH_g h=h_i^{-1}\cH_g h_i$.

Now, fix $\psi \in \cO_{\Mod(S_g)}(\varphi)\cap \cH_g$. As $\cO_{\Mod(S_g)}(\varphi)$, $\psi=h \varphi  h^{-1}$ for some $h\in \Mod(S_g)$. Then, there is a unique $i\in \{1,2,\cdots,k\}$ such that $h^{-1}\cH_g h=h_i^{-1}\cH_g h_i$. Since $\cH_g=(hh_i^{-1})\cH_g (hh_i^{-1})^{-1}$, by \reflem{trivialNorm} $hh_i^{-1}\in \cH_g$ and so $h=fh_i$ for some $f\in \cH_g$. Therefore, 
$$\psi=h \varphi  h^{-1}=(fh_i)\varphi(fh_i)^{-1}=fh_i\varphi h_i^{-1}f^{-1}=f\varphi^{h_i}f^{-1}$$
and $\psi \in \cO_{\cH_g}(\varphi^{h_i})$. Thus, 
$$\cO_{\Mod(S_g)}(\varphi)\cap \cH_g=\bigcup_{j=1}^k \cO_{\cH_g}(\varphi^{h_j}).$$ 

Combining \refthm{discretSpec} with this observation, we can get the desired result. 
\end{proof}

\subsection{Spectra for outer automorphisms}
We define the \emph{spectrum} $\Spec
(\varphi)$ for an outer automorphism $\varphi$ to be the set 
$$\{\lambda_\psi: A(\psi)=\varphi \text{ and $\psi$ is pseudo-Anosov}\}.$$
Now, we present the topology of spectra.

\begin{lem}\label{Lem:specId}
Let $g\geq 2$. The spectrum $\Spec(\id_{\Out(F_g)})$ for the trivial outer automorphism $\id_{\Out(F_g)}$ is a non-empty unbounded subset of $\RR$
\end{lem}
\begin{proof}
By \cite{Luft78}, it is shown that the kernel of $A$ is generated by the Dehn twists about meridians. Note that $\cH_g$ contains a pseudo-Anosov element $f$. 
For any meridian curve $\alpha$, we can choose $k\in \NN$ so that $\alpha$ and $f^k(\alpha)$ fill $S_g$ with $i(\alpha,f^k(\alpha))>2$. 
Then, by the Thurson's consturction, we can see that $T_\alpha \circ T_{f^k(\alpha)}$ is a pseudo-Anosov element. 
Moreover, as $T_\alpha,T_{f^k(\alpha)}\in \ker(A)$, $T_\alpha \circ T_{f^k(\alpha)}\in \ker(A)$ and so $(T_\alpha \circ T_{f^k(\alpha)})^m\in \ker(A)$ for all $m\in\ZZ$. 
This shows that $\Spec(\id_{\Out(F_g)})$ is a non-empty unbounded subset of $\RR$.
\end{proof}

\begin{lem}\label{Lem:specInR}
Let $\varphi$ be an outer automorphism in $\Out(F_g)$, $g\geq 2$. Then, $\Spec(\varphi)$ is a closed and discrete subset of $\RR$.
\end{lem}
\begin{proof}
It follows from \refthm{discretSpecH}.

\end{proof}

We show the following theorem to see that the stretch factor of a fully irreducible element  $\varphi$  in $\Out(F_g)$, $g>1$, provides a lower bound of $\Spec(\varphi)$.
\begin{lem} \label{Lem:LimitSqrtSum}
    For sequences $\{a_n\}, \{b_n\}$ of positive real numbers, if $$\lim_{n \to \infty} \sqrt[n]{a_n} = \lim_{n \to \infty} \sqrt[n]{b_n} = \lambda > 0,$$
    then we have $\lim_{n \to \infty} \sqrt[n]{a_n + b_n} = \lambda$.
\end{lem}

\begin{proof}
Choose $\epsilon$ with $0<\epsilon <\lambda$. Note that $\{\sqrt[n]{2}\}_{n=2}^\infty$ is a decreasing sequence converging to $1$. Hence, we can take a number $N_1>1$ in $\NN$ so that $$1<\sqrt[n]{2}<1+\frac{\epsilon}{2\lambda}$$ for all $n>N_1$. 
Then, as $\lim_{n \to \infty} \sqrt[n]{a_n} = \lim_{n \to \infty} \sqrt[n]{b_n}=\lambda$, there is a number $N_2$ in $\NN$  satisfying the following:
\begin{itemize}
    \item $N_2>N_1$,
    \item $\lambda-\epsilon/4 \leq \sqrt[n]{a_n} \leq\lambda+\epsilon/4$ for all $n>N_2$, and 
    \item $\lambda-\epsilon/4 \leq \sqrt[n]{b_n} \leq\lambda+\epsilon/4$ for all $n>N_2$.
\end{itemize}

Now, fix $n$ with $n>N_2$. We have that
 $$(\lambda-\epsilon/4)^n \leq a_n \leq(\lambda+\epsilon/4)^n\text{ and }(\lambda-\epsilon/4)^n \leq b_n \leq(\lambda+\epsilon/4)^n.$$
 Then, 
$$2(\lambda-\epsilon/4)^n \leq a_n+b_n\leq 2(\lambda+\epsilon/4)^n,$$ and so
$$\sqrt[n]{2}(\lambda-\epsilon/4) \leq \sqrt[n]{a_n+b_n}\leq \sqrt[n]{2}(\lambda+\epsilon/4).$$
Equivalently,
$$\sqrt[n]{2}(\lambda-\epsilon/4)-\lambda \leq \sqrt[n]{a_n+b_n}-\lambda\leq \sqrt[n]{2}(\lambda+\epsilon/4)-\lambda.$$

First, we consider the left side:
\[
\sqrt[n]{2}(\lambda-\epsilon/4)-\lambda = (\sqrt[n]{2}-1)\lambda-\frac{\sqrt[n]{2}}{4}\epsilon >-\frac{\sqrt[n]{2}}{4}\epsilon >-\epsilon
\]
as $1<\sqrt[n]{2}<2$.

Next,  we consider the right side:
\[
\sqrt[n]{2}(\lambda+\epsilon/4)-\lambda=(\sqrt[n]{2}-1)\lambda+\frac{\sqrt[n]{2}}{4}\epsilon<\frac{\epsilon}{2\lambda}\cdot \lambda+\frac{1}{2}\epsilon=\epsilon.
\]
as $\sqrt[n]{2}-1<\epsilon/2\lambda$ and $1<\sqrt[n]{2}<2$.
Therefore, for every $n>N_2$, $$|\sqrt[n]{a_n+b_n}-\lambda|<\epsilon.$$ Thus, $\lim_{n\to \infty }\sqrt[n]{a_n+b_n}=\lambda$.

\end{proof}

\begin{thm}\label{Thm:lowerBound}
For every pseudo-Anosov mapping class $\varphi \in \mathcal{H}_g$, if $A(\varphi)$ is fully irreducible, then we have $ \mu_{A(\varphi)} \leq \lambda_\varphi$. In particular,  for every fully irreducible element $\psi$ in $\Out(F_g)$, $\Spec(\psi)$ is bounded below by $\mu_\psi$.
\end{thm}

\begin{proof}
Let $\delta_1, \dots, \delta_g$ be a cut system of $V_g$, and let $R$ be an embedded rose dual to this cut system.
Then there exists a retraction $r: V_g \to R$ such that $r(\delta_j)$ is a point for all $j$.
If $v$ is the vertex of $R$, then $r^{-1}(R \setminus \{v\}) \cap \Sigma_g$ is a disjoint union of open neighborhoods of $\delta_1, \dots, \delta_g$, that is, $$r^{-1}(R \setminus \{v\}) \cap \Sigma_g = \bigsqcup_{j = 1}^g N(\delta_j)$$ for some open neighborhoods $N(\delta_1), \dots, N(\delta_g) \subset \Sigma_g$ of $\delta_1, \dots, \delta_g$, respectively.

Choose an embedded loop $\hat\gamma: S^1 \to \Sigma_g$ that moves along a non-meridian simple closed curve $\gamma$.
Note that $A(\varphi)$ does not permute $r\hat\gamma$ periodically.
For each $n \geq 1$, let $f_n$ denote a self-homeomorphism representative to $\varphi^n$ such that the number of the components of $f_n(\gamma) \cap N(\delta_j)$ is minimal, in other words, $$\lvert \pi_0(f_n(\gamma) \cap N(\delta_j)) \rvert = \lvert f_n(\gamma) \cap \delta_j \rvert = i(\varphi^n([\gamma]), [\delta_j])$$ for all $j$.
By \reflem{LimitSqrtSum} and \cite[Theorem 14.24]{FarbMargalit12}, the stretch factor of $\varphi$ is given by $$\lambda_\varphi = \lim_{n \to \infty} \sqrt[n]{i(\varphi^n([\gamma]), [\delta_1])} = \lim_{n \to \infty} \sqrt[n]{\sum_{j = 1}^g \lvert f_n(\gamma) \cap \delta_j \rvert }.$$

Notice that the edge length of a loop of $R$ is equal to the word length of the word corresponding to the loop.
For each $n$, if $\rho_n$ is a shortest loop homotopic to $rf_n\hat\gamma$, then the edge length of $\rho_n$ is at most $\lvert \pi_0((rf_n\hat\gamma)^{-1}(R \setminus \{v\})) \rvert$, the number of times $rf_n\hat\gamma$ passes through the edges.
If $\|\rho_n\|_R$ denotes the edge length of $\rho_n$, then one deduces
\begin{align*}
    \|\rho_n\|_R &\leq \lvert \pi_0((rf_n\hat\gamma)^{-1}(R \setminus \{v\})) \rvert = \left\lvert \pi_0\left((f_n\hat\gamma)^{-1}\left( \bigsqcup_{j = 1}^g N(\delta_j) \right)\right) \right\rvert \\
    &= \sum_{j = 1}^g \lvert \pi_0((f_n\hat\gamma)^{-1}(N(\delta_j))) \rvert = \sum_{j = 1}^g \lvert \pi_0 (f_n(\gamma) \cap N(\delta_j)) \rvert \\
    &= \sum_{j = 1}^g \lvert f_n(\gamma) \cap \delta_j \rvert.
\end{align*}

Following Bestvina--Handel \cite[Remark 1.8]{BestvinaHandel92}, because $A(\varphi)$ is irreducible, we have $\mu_{A(\varphi)} = \lim_{n \to \infty} \sqrt[n]{\|\rho_n\|_R}$.
Therefore, we get that  $$\mu_{A(\varphi)} = \lim_{n \to \infty} \sqrt[n]{\|\rho_n\|_R} \leq \lim_{n \to \infty} \sqrt[n]{\sum_{j = 1}^g \lvert f_n(\gamma) \cap \delta_j \rvert} = \lambda_\varphi.$$
Thus, the statement holds.
\end{proof}

\section{The outer automorphism group of the free group of rank two}\label{Sec:OutF2}
In this section, we briefly discuss some facts  without proof and introduce some conventions related to $\Out(F_2)$, which are used later. 

\subsection{$\Out(F_2)$ is $\GL(\ZZ)$} \label{Sec:OutToGL}
Let $F_2$ be the free group of rank two. We denote the homomorphism from $F_2$ to the abelianization $F_2^{ab}\cong \ZZ^2$ by $\theta^{ab}$. Then, for each automorphism $\psi\in \Aut(F_2)$, there is a unique automorphism $\Theta^{ab}(\psi)$ in $\Aut(\ZZ^2)$ such that $\Theta^{ab}(\psi) \circ \theta^{ab}=\theta^{ab} \circ \psi$. Then, $\Theta^{ab}$ is an epimorphism from $\Aut(F_2)$ to $\Aut(\ZZ^2)$ with $\ker(\Theta^{ab})=\Inn(F_2)$. Hence, there is the isomorphism $\tilde{\Theta}^{ab}$ from $\Out(F_2)$ to $\Aut(\ZZ^2)\cong \Out(\ZZ^2) \cong \GL(\ZZ)$.
This is a well known Nielsen's result. One can also find an alternative explanation for the isomorphism $\tilde{\Theta}^{ab}$ in \cite[Section~5.3]{BestvinaBuxMargalit07}.

Now, set $\{x_1,x_2\}$ to be the free generating set of $F_2$. We may assume that $\theta^{ab}(x_1)=(1,0)$ and $\theta^{ab}(x_2)=(0,1)$. We define the autormorphisms $\sigma_{12}$, $\sigma_{21}$, $\epsilon_1$, and  $\epsilon_2$ as follows:
$$\sigma_{ij}(x_j)=x_j \text{ and } \sigma_{ij}(x_i)=x_ix_j$$
and 
$$\epsilon_i(x_j)=x_j \text{ and } \epsilon_i(x_i)=x_i^{-1}$$
where $i\neq j$. Then, $\{\sigma_{12},\sigma_{21},\epsilon_1,\epsilon_2\}$ is a generating set of $\Out(F_2)$. 
Also, we obtain that  
$$\Theta^{ab}(\sigma_{12})=\begin{pmatrix}1&1\\0&1\end{pmatrix} \text{ and } \Theta^{ab}(\sigma_{21})=\begin{pmatrix}1&0\\1&1\end{pmatrix} $$
and 
$$\Theta^{ab}(\epsilon_1)=\begin{pmatrix}-1&0\\0&1\end{pmatrix} \text{ and }\Theta^{ab}(\epsilon_2)=\begin{pmatrix}1&0\\0&-1\end{pmatrix}.$$

\subsection{$\Out(F_2)$ is $\Mod^{\pm}(S_{1,1})$}\label{Sec:OutToMod}
$F_2^{ab}=\ZZ^2$ acts on $\RR^2$ by translations in the standard way. Then, $F_2^{ab}$ can be considered as the fundamental group of the torus $S_1=\RR^2/\ZZ^2$. Also, we may think of $F_2$ as the fundamental group of once punctured torus $S_{1,1}=(\RR^2\setminus (1/2,1/2)+\ZZ^2)/\ZZ^2$. Then, $\theta^{ab}$ is the homomorphism induced from the inclusion map $i:S_{1,1}\to S_1$. 
Note that $i$ induces the isomorphism $i_\#$ from $H_1(S_{1,1};\ZZ)$ to $H_1(S_1;\ZZ)$. Then, via the isomorphism $i_\#$, we can see that 
$$\Mod^{\pm}(S_{1,1})=\Mod^{\pm}(S_1)=\GL(\ZZ)=\Out(F_2).$$

Note that each primitive element $(q,p)$ in $H_1(S_{1,1};\ZZ)$ or $H_1(S_1;\ZZ)$ is uniquely associated with an essential simple closed curve $\alpha_{p/q}$ in $S_{1,1}$ or $S_1$, respectively, a full lift of which is the line with slop $p/q$  in $\RR^2\setminus (1/2,1/2)+\ZZ^2$ or $\RR^2$, respectively, where $1/0=\infty$. In this reason, the set of essential simple closed curves in $S_{1,1}$ or $S_1$ is identified with the rational point $[q:p]$ in the real projective line $\RP^1$.

\subsection{$\GL(\ZZ)$ in $\Isom(\HH^2)$}

Now, we consider the action of $\GL(\ZZ)$ on $\RP^1$. By identification, the action of $\begin{pmatrix}a&c\\b&d\end{pmatrix}
\in \GL(\ZZ)$ on $\RP^1$ is defined as 
$$\begin{pmatrix}a&c\\b&d\end{pmatrix}\cdot[x:y]=[ax+cy:bx+dy].$$ Then, we consider the hyperbolic plane $\HH^2$ as the upper half plane in the Riemann sphere $\hat{\CC}=\CC \cup \{\infty\}$ and we think of $\hat{\RR}=\RR \cup \{\infty \}$ as the boundary of $\HH^2$. Now, we identify $\hat {\CC}$ with $\CP^1$ via the map $z\to [1:z]$ where $1/0=\infty$. Under this identification, $\RP^1$ is identified with $\hat{\RR}$ and each simple closed curve $\alpha_{p/q}$ is corresponded to the point $p/q$ in $\hat{\QQ}=\QQ\cup\{\infty\}$. Then, the action of $\begin{pmatrix}a&c\\b&d\end{pmatrix}$ on $\RP^1$ can be written as 
\[\begin{pmatrix}a&c\\b&d\end{pmatrix}\cdot z=\frac{b+dz}{a+cz}
\]
for $z\in \hat{\RR}$. However, the action of $\GL(\ZZ)$ on $\hat{\RR}$ is not faithful. As the kernel of the action is $\{I,-I\}$, $\PGL{\ZZ}$ faithfully acts on $\hat{\RR}$. Moreover, this action can be extended on $\HH^2$ as follows:
if $ad-bc=1$, then
$\begin{pmatrix}a&c\\b&d\end{pmatrix}$ acts on $\HH^2$ as  the linear fractional transformation \[\frac{b+dz}{a+cz}.\] Otherwise, $ad-bc=-1$ and so $\begin{pmatrix}a&c\\b&d\end{pmatrix}$ acts on $\HH^2$ as \[\frac{b+d\bar{z}}{a+c\bar{z}}.\]
Therefore, we can see that $\PGL{\ZZ}$ considered as a subgroup of the isometry group $\Isom(\HH^2)$ of the hyperbolic plane, including the orientation reversing maps. In fact, $\Isom(\HH^2)=\PGL{\RR}$. In this perspective, $\Theta^{ab}(\sigma_{12})$ is corresponded to the parabolic isometry $z\mapsto z/(1+z)$ fixing $0$. Likewise, $\Theta^{ab}(\sigma_{21})$ is corresponded to the parabolic isometry $z\mapsto z+1$ fixing $\infty$. 

\subsection{The classification of isometries}
Let $\varphi$ be an isometry of $\HH^2$. If $\varphi$ is orientation preserving, then $\varphi$ falls into three cases: elliptic, parabolic, hyperbolic. If $\varphi$ is orientation reversing, then $\varphi$ falls into two cases: a reflection, a glide-reflection. If $\varphi$ is a reflection, then $\varphi$ is conjugated to the isometry $z\mapsto -\bar z$ and so $\Fix{\varphi}$ is a geodesic which is the reflection axis. If $\varphi$ is a glide-reflection, then $\varphi$ is conjugated to the isometry $g:z\mapsto -e^\alpha \bar z$ for some positive real number $\alpha \in \RR$. Hence, there is a unique invariant geodesic on which $\varphi$ acts as a translation with translation length $\alpha$. Note that $g$ is associated with $A_g=\begin{pmatrix}
e^{-\alpha/2}&0\\0&-e^{\alpha/2}
\end{pmatrix}$. Hence, if $A_\varphi$ is the element in $\PGL{\RR}$ associated with $\varphi$, then $\alpha$  equals to $2| \log|\lambda||$ where $\lambda$ is an eigenvalues of $A_\varphi$. See \cite{Scott83} for the more detailed exposition.

\subsection{The curve complex of $S_{1,1}$}

The curve graph $\cC(S_{1,1})$ of $S_{1,1}$ is a graph whose vertex set is the essential simple closed curves $\hat \QQ$ of $S_{1,1}$ and such that there is an edge between two vertices $b/a$ and $d/c$ if and only if $|ad-bc|=1$. Under the previous identification, the curve graph $\cC(S_{1,1})$ of $S_{1,1}$ can be identified with the \emph{Farey graph} $\cF$ in $\closure{\HH^2}$. To see this, let $\ell$ be the geodesic connecting $0=0/1$ and $\infty=1/0$. Then, for each $M=\begin{pmatrix}
 a&c\\b&d   
\end{pmatrix}\in \PGL{\ZZ}$, $M$ maps $\ell$ to the geodesic connecting $b/a=M\cdot 0/1$ and $d/c=M\cdot 1/0$. As $|ad-bc|=1$, we can think of $\cF$ as the embedding of $\cC(S_{1,1})$ into $\closure{\HH^2}$. 

\subsection{Standard forms and the syllable length}\label{Sec:stdForm}
 
Let $M$ be an element in $\GL(\ZZ)$. If $M$ is a hyperbolic isometry as an element in $\PGL{\ZZ}$, then there is a invariant axis $\ell_M$ whose end points in $\hat{\RR}$ are  the slopes of two eigenvectors of $M$. Now, we orient $\ell_M$ so that $M$ acts on $\ell_M$ as a positive translation. Note that the Farey graph $\cF$ gives a ideal triangulation of $\HH^2$, so-called the \emph{Farey triangulation}. 
Also, we call each ideal triangle in the Farey triangulation a \emph{Farey triangle}.  
Then, the collection of Farey triangles intersecting $\ell_M$ gives rise to a bi-infinite sequence $\{f_i\}_{i\in \ZZ}$ of consecutive Farey triangles such that $f_i$ and $f_{i+1}$ share only one geodesic edge and there is a positive integer $w(M)$ in $\NN$ such that $M(f_i)=f_{i+w(M)}$ for all $i\in \ZZ$. We denote the sequence $\{f_i\}_{i\in \ZZ}$ by $\Sigma_M$ and we called $\Sigma_M$ the \emph{ladder} of $M$. See \cite{FloydHatcher82},\cite{ShinStrenner15}, and \cite{BaikKimKwakShin21} for the detailed description of ladders. For each $f_i$ in $\Sigma_M$, we call the edges of $f_i$ intersecting $\ell_M$ the \emph{cut sides} of $f_i$. We define the \emph{cutting sequence} $\{U_i\}_{i\in\ZZ}$ of $\Sigma_M$ that is a bi-infinite sequence of letters $L$ and $R$, that is, $U_i\in \{L,R\}$, and is constructed as follows. If two cut sides of $f_i$ share a vertex in the left side of $\ell_M$, then $U_i=L$. If two cut sides of $f_i$ share a vertex in the right side of $\ell_M$, then $U_i=R$. As $M(f_i)=f_{i+w(M)}$ for all $i\in \ZZ$, the cutting sequence $\{U_i\}_{i\in\ZZ}$ is also periodic, that is, $U_i=U_{i+w(M)}$ for all $i \in \ZZ$. Hence, there is $i_0$ in $\ZZ$ and a finite sequence $\{a_j\}_{j=1}^d$ of positive integers such that these satisfy the following:
\begin{enumerate}
    \item $b_d=w(M)$
    \item $U_{i_0}=R$ and $U_{i_0-b_d-1}=L$,
    \item for each $j\in\{0,1,2,\cdots,d-1\}$, $U_{i_0+b_j}=U_{i_0+b_j+n}$ for all $0\leq n < a_{j+1}$, and 
    \item $\{U_{i_0+b_j-1},U_{i_0+b_j}\}=\{L,R\}$ for all $j\in \{1,2,\cdots,d-1\}$.
\end{enumerate}
where $a_0=0$ and $b_j=\sum_{k=0}^{j}a_k$ for $j\in \{0,1,2,\cdots,d\}$. 
Note that $d$ is even. We call the finite sequence $\{a_j\}_{j}^d$ the \emph{type} of the ladder $\Sigma_M$ and $i_0$ a \emph{base index} for the type. Note that the type of a ladder is well defined uniquely up to cyclic permutations.
This implies that we can find an element $S(M)$ in $\cO_{\GL(\ZZ)}(M)$ in the following form:
$$S(M)=\pm \begin{pmatrix}
1&1\\0&1    
\end{pmatrix}^{a_1}\begin{pmatrix}
1&0\\1&1    
\end{pmatrix}^{a_2}\cdots \begin{pmatrix}
1&1\\0&1    
\end{pmatrix}^{a_{d-1}}\begin{pmatrix}
1&0\\1&1    
\end{pmatrix}^{a_d}.$$
Here, the sign of $S(M)$ is chosen to be the sign of the trace of $M$. We call $S(M)$ the \emph{standard form} of $M$. The standard form is well defined uniquely up to conjugation. Also, we call $d$ be the \emph{stable  syllable length} of $M$ and denote it by $d(M)$.

When $M$ is a glide-reflection as an element in  $\PGL{\ZZ}$, there is a unique invariant geodesic $\ell_M$. We can orient $\ell_M$ so that $M$ acts on $\ell_M$ as a positive translation. Then, we can define a \emph{ladder}, the type of the ladder, and the standard form of $M$ in a similar way. To see this, consider the hyperbolic isometry $M^2$ whose invariant axis is $\ell_M$. Also, $M^2$ acts on $\ell_M$ as a positive translation. Then, there is the ladder $\Sigma_{M^2}=\{f_i\}_{i\in\ZZ}$, the cutting sequence $\{U_i\}_{i\in \ZZ}$, and the type $\{a_j\}_{j=1}^{d(M^2)}$ of the ladder with base index $i_0$ for $M^2$. Observe that $M(f_{i_0})=f_{i_0+b_M}$, $U_{i_0}=R$ and $U_{i_0+b_M}=L$ where $b_M=\sum_{j=1}^{d(M^2)/2}a_j$. Therefore, $a_j=a_{j+d(M^2)/2}$ for all $j\in\{1,2,\cdots, d(M^2)/2\}$ and $d(M^2)/2$ is odd. Moreover, we can take an element $S(M)$ in $\cO_{\GL(\ZZ)}(M)$ in the following form:
$$S(M)=\pm \begin{pmatrix}
1&1\\0&1    
\end{pmatrix}^{a_1}\begin{pmatrix}
1&0\\1&1    
\end{pmatrix}^{a_2}\cdots \begin{pmatrix}
1&0\\1&1    
\end{pmatrix}^{a_{d(M^2)/2-1}}\begin{pmatrix}
1&1\\0&1    
\end{pmatrix}^{a_{d(M^2)/2}}\begin{pmatrix}
0&1\\1&0    
\end{pmatrix}.$$
Note that the last factor $\begin{pmatrix}
0&1\\1&0    
\end{pmatrix}$ comes from the fact that $U_{i_0}=R$ and $U_{i_0+b_M}=L$ and that $$
\begin{pmatrix}
0&1\\1&0    
\end{pmatrix}
\begin{pmatrix}
1&1\\0&1    
\end{pmatrix}
\begin{pmatrix}
0&1\\1&0    
\end{pmatrix}=\begin{pmatrix}
1&0\\1&1    
\end{pmatrix}.$$
Here, the sign of $S(M)$ is chosen to be the sign of the trace of $M$.
Therefore, $S(M^2)=\pm S(M)^2$. We define the \emph{standard form} of $M$ to be $S(M)$.
Similarly, we define the \emph{ladder} $\Sigma_M$ of $M$ to be $\Sigma_{M^2}$ and the cutting sequence for $\Sigma_M$ to be $\{U_i\}_{i\in \ZZ}$. The \emph{type} of $\Sigma_M$ is defined as the finite sequence $\{a_j\}_{j=1}^{d(M^2)/2}$ of positive integers and the \emph{stable syllable length} $d(M)$ of $M$ is $d(M^2)/2$.

\section{On traces}\label{Sec:trace}

In this section, we investigate some properties of standard forms to estimate the stretch factor of an fully irreducible element in $\Out(F_2)$. In particular, the comparison of the trace and the anti-trace of a given matrix plays a crucial role in the proof of the main theorem, \refthm{almostMinimal}.  
\subsection{Traces and stretch factors}\label{Sec:trAndStr}

Let $M$ be an element in $\GL(\ZZ)$. Assume that $M$ is a hyperbolic isometry or glide reflection as an element in $\PGL{\ZZ}$. Then, $M$ has two distinct real eigenvalues $\lambda_1$ and $\lambda_2$ and two eigenvectors $v_1=(x_1,x_2)$ and $v_2=(y_1,y_2)$, respectively. Note that $\tr(M)=\lambda_1+\lambda_2$ and $\det(M)=\lambda_1\lambda_2$. Hence, if $M$ is a hyperbolic isometry, then $\det(M)=\lambda_1\lambda_2=1$ and so $\lambda_2=\lambda_1^{-1}$. Therefore, $\lambda_1$ and $\lambda_2$ have the same sign, and  $\tr(M)=\lambda_1+\lambda_2=\lambda_1+\lambda_1^{-1}$. When $M$ is a glide-reflection,  $\det(M)=\lambda_1\lambda_2=-1$ and so $\lambda_2=-\lambda_1^{-1}$. Hence, the signs of $\lambda_1$ and $\lambda_2$ are different, and $\tr(M)=\lambda_1+\lambda_2=\lambda_1-\lambda_1^{-1}$.
Therefore, the parallel lines with the same slope of $v_1$ on $\RR^2\setminus \ZZ^2$ give rise to a  foliation $\tilde{\sF_1}$ preserved by $M$. Then, there is a standard transverse measure $\textbf{m}_1$ in  $\tilde{\sF_1}$ induced from $|y_1dx+y_2dy|$ and so $M\cdot(\tilde{\sF_1},\textbf{m}_1)=(\tilde{\sF_1},|\lambda_2|\textbf{m}_1)$. Likewise, the parallel lines with the same slope of $v_1$ on $\RR^2\setminus \ZZ^2$ give rise to a foliation $\tilde{\sF_2}$ preserved by $M$. Also, there is a standard transverse measure $\textbf{m}_2$ induced from the one form $|x_1dx+x_2dy|$ and so $M\cdot (\tilde{\sF_2},\textbf{m}_2)=(\tilde{\sF_2}, |\lambda_1|\textbf{m}_2)$. Therefore,  $(\tilde{\sF_1}, \textbf{m}_1)$ and  $(\tilde{\sF_2}, \textbf{m}_2)$ descend to $S_{1,1}=(\RR^2\setminus \ZZ^2)/\ZZ^2$ and give rise to two transverse measured foliations $(\sF_1, \textbf{m}_1)$ and  $(\sF_2,\textbf{m}_2)$ on $S_{1,1}$. Then, $M\cdot (\sF_1,\textbf{m}_1)=(\sF_1,|\lambda_2|\textbf{m}_1)$ and $M\cdot (\sF_2,\textbf{m}_2)=(\sF_2,|\lambda_1|\textbf{m}_2)$. Thus, $ \max \{|\lambda_1|, |\lambda_2|\}>1$ is the stretch factor of $M$ as a pseudo-Anosov element in $\Mod^{\pm}(S_{1,1})$.

\subsection{The trace formula}
We first intend to describe the trace of the form $$W_x(m_1,m_2,\cdots, m_{2d})=\prod_{i=1}^d A_x^{m_{2i-1}}B_x^{m_{2i}}$$ for $$A_x=\begin{pmatrix} 1 & x \\ 0 & 1 \end{pmatrix} \text{ and } B_x=\begin{pmatrix} 1 & 0 \\ x & 1 \end{pmatrix} \in \SL(\ZZ[x])$$ and for any $d \in \NN$ and $m_1, \dots, m_{2d} \in \mathbb{Z}$. For each $k \in \{ 2, 4, \dots, 2d \}$, write $$J^{k}_d := \{ (j_1, \dots, j_{k}) \mid 1 \leq j_1 < \dots < j_{k} \leq 2d ~\text{and}~ j_{\ell+1} - j_\ell ~\text{is}~\text{odd} ~\forall \ell \}$$ and $$c_d(k) =\sum_{(j_1, \dots, j_{k}) \in J^{k}_d}m_{j_1} \dots m_{j_{k}}.$$

\begin{lem}\label{Lem:trForm}
For every $d\in \NN$ and $m_1, \dots, m_{2d} \in \ZZ$, we have
$$\tr(W_x(m_1, \dots, m_{2d})) = 2 + \sum_{i = 1}^d c_{d}(2i) x^{2i}$$
\end{lem}
\begin{proof}
For $d = 1$, calculating $W_x$, we have
\begin{align*}
    W_x &= A_x^{m_1} B_x^{m_2} = \begin{pmatrix}1 & m_1x \\ 0 & 1\end{pmatrix} \begin{pmatrix}1 & 0 \\ m_2x & 1\end{pmatrix} = \begin{pmatrix} 1 + m_1m_2x^2 & m_1x \\ m_2x & 1 \end{pmatrix}.
\end{align*}
So the trace of $W_x$ is $2 + m_1m_2x^2$, which satisfies the statement for $d=1$.

We claim $W_x$ can be written as $$W_x = \begin{pmatrix}
    1 + \sum_{i = 1}^d c_{d, o}(2i) x^{2i} & \sum_{i = 1}^d c_{d, o}(2i-1) x^{2i-1} \\
    \sum_{i = 1}^d c_{d, e}(2i-1) x^{2i-1} & 1 + \sum_{i = 1}^{d} c_{d, e}(2i) x^{2i}
\end{pmatrix}$$
where for each $k$,
    $$c_{d, o}(k) = \sum_{(j_1, \dots, j_{k}) \in J_{d, o}^{k}} m_{j_1} \dots m_{j_{k}} \quad \text{and} \quad c_{d, e}(k) = \sum_{(j_1, \dots, j_{k}) \in J_{d, e}^{k}} m_{j_1} \dots m_{j_{k}}$$
and
    $$J_{d, o}^{k} = \{ (j_1, \dots, j_{k}) \in J^{k}_d \mid j_1 ~\text{is}~\text{odd} \} ~\text{and}~
    J_{d, e}^{k} = \{ (j_1, \dots, j_{k}) \in J^{k}_d \mid j_1 ~\text{is}~\text{even} \}.$$
Define $J_{d,o}^k$ (\textit{resp.,} $J_{d, e}^k$) as an empty set for all $k \leq 0$ or $k > 2d$ (\textit{resp.,} $k \leq 0$ or $k \geq 2d$).

To prove the claim, we are going to use an induction on $d$.
The initial step is already shown above.
For the induction step, consider $W_x' = W_x A_x^{m_{2d+1}} B_x^{m_{2d+2}}$ for some nonzero integers $m_{2d+1}, m_{2d+2}$.
For all $1 \leq i \leq d+1$, we observe the following list of decomposition:
\begin{enumerate}
    \item $J_{d+1, o}^{2i} = J_{d, o}^{2i} \sqcup ( J_{d, o}^{2i-1} \times \{ 2d+2 \} ) \sqcup ( J_{d, o}^{2i-2} \times \{ (2d+1, 2d+2) \} )$;
    \item $J_{d+1, o}^{2i-1} = J_{d, o}^{2i-1} \sqcup ( J_{d, o}^{2i-2} \times \{ 2d+1 \} )$;
    \item $J_{d+1, e}^{2i} = J_{d, e}^{2i} \sqcup ( J_{d, e}^{2i - 1} \times \{ 2d+1 \} )$;
    \item $J_{d+1, e}^{2i - 1} = J_{d, e}^{2i - 1} \sqcup ( J_{d, e}^{2i - 2} \times \{ 2d+2 \} ) \sqcup ( J_{d, e}^{2i-3} \times \{ (2d+1, 2d+2) \} )$.
\end{enumerate}
For each $1 \leq i \leq d+1$, applying the above to $c_{d, o}$'s and $c_{d, e}$'s, we get
\begin{enumerate}
    \item $c_{d+1, o}(2i) = c_{d, o}(2i) + c_{d, o}'(2i-1) m_{2d+2} + c_{d, o}'(2i-2) m_{2d+1} m_{2d+2}$,
    \item $c_{d+1, o}(2i-1) = c_{d, o}(2i-1) + c_{d, o}'(2i-2) m_{2d+1}$,
    \item $c_{d+1, e}(2i) = c_{d, e}(2i) + c_{d, e}'(2i-1) m_{2d+1}$,
    \item $c_{d+1, e}(2i-1) = c_{d, e}(2i-1) + c_{d, e}'(2i-2) m_{2d+2} + c_{d, e}'(2i-3) m_{2d+1} m_{2d+2}$
\end{enumerate}
where $c_{d, *}'(k) = \begin{cases} c_{d, *}(k) & \text{for} ~ 1 \leq k \leq 2d, \\ 1 & \text{for}~ k = 0, \\ 0, & \text{otherwise.} \end{cases}$

Note $W_x'$ is expressed as
\begin{align*}
    W_x' &= W_x A_x^{m_{2d+1}} B_x^{m_{2d+2}} \\
    &= \left(\begin{smallmatrix}
        1 + \sum_{i = 1}^d c_{d,o}(2i) x^{2i} & \sum_{i = 1}^d c_{d,o}(2i-1) x^{2i-1} \\
        \sum_{i = 1}^d c_{d,e}(2i-1) x^{2i-1} & 1 + \sum_{i = 1}^{d} c_{d,e}(2i) x^{2i}
    \end{smallmatrix}\right)
    \left(\begin{matrix}
        1 + m_{2d+1} m_{2d+2} x^2 & m_{2d+1} x \\
        m_{2d+2} x & 1
    \end{matrix}\right).
\end{align*}
Let us now compute each entry of $W_x'$.
The $(1,1)$-entry of $W_x'$ is
\begin{align*}
    &\left(1 + \sum_{i = 1}^d c_{d,o}(2i) x^{2i}\right)(1 + m_{2d+1} m_{2d+2} x^2) + \sum_{i = 1}^d c_{d,o}(2i-1) m_{2d+2} x^{2i} \\
    = & 1 + ( c_{d,o}(2) + c_{d,o}(1) m_{2d+2} + m_{2d+1} m_{2d+2} ) x^2 \\
    & \quad + \sum_{i=2}^d (c_{d,o}(2i) + c_{d,o}(2i-1) m_{2d+2} + c_{d,o}(2i-2) m_{2d+1} m_{2d+2} ) x^{2i} \\
    & \quad + c_{d,o}(2d) m_{2d+1} m_{2d+2} x^{2d+2} \\
    = & 1 + c_{d+1, o}(2) x^2 + \left( \sum_{i=2}^d c_{d+1, o}(2i) x^{2i} \right) + c_{d+1, o}(2d+2s) x^{2d+2} \\
    = & 1 + \sum_{i=1}^{d+1} c_{d+1, o}(2i) x^{2i}.
\end{align*}
The $(1, 2)$-entry of $W_x'$ is
\begin{align*}
    &\left( 1 + \sum_{i = 1}^d c_{d,o}(2i) x^{2i} \right) m_{2d+1} x + \sum_{i = 1}^d c_{d,o}(2i-1) x^{2i-1} \\
    = & ( c_{d, o}(1) + m_{2d+1} ) x + \sum_{i = 2}^d ( c_{d,o}(2i-1) + c_{d,o}(2i-2) m_{2d+1} ) x^{2i-1} \\
    & \quad + c_{d,o}(2d) m_{2d+1} x^{2d+1} \\
    = & c_{d+1, o}(1) x + \left( \sum_{i=2}^d c_{d+1, o}(2i-1) x^{2i-1} \right) + c_{d+1, o}(2d+1) x^{2d+1} \\
    = & \sum_{i = 1}^{d+1} c_{d+1, o}(2i-1) x^{2i-1}.
\end{align*}
The $(2, 1)$-entry of $W_x'$ is
\begin{align*}
    & \left( \sum_{i = 1}^d c_{d,e}(2i-1) x^{2i-1} \right) ( 1 + m_{2d+1} m_{2d+2} x^2 ) + \left( 1 + \sum_{i = 1}^{d} c_{d,e}(2i) x^{2i} \right) m_{2d+2} x \\
    = & ( c_{d,e}(1) + m_{2d+2} ) x \\
    & \quad + \left( \sum_{i = 2}^d ( c_{d, e}(2i-1) + c_{d, e}(2i-2) m_{2d+2} + c_{d, e}(2i-3) m_{2d+1} m_{2d+2} ) x^{2i-1} \right) \\
    & \quad + ( m_{2d+2} + c_{d,e}(2d-1) m_{2d+1} m_{2d+2} ) x^{2d+1} \\
    = & c_{d+1, e}(1) x + \left( \sum_{i=2}^d c_{d+1, e}(2i-1) x^{2i-1} \right) + c_{d+1, e}(2d+1) x^{2d+1} \\
    = & \sum_{i=1}^{d+1} c_{d+1, e}(2i-1) x^{2i-1}.
\end{align*}
Since $c_{d, e}(2d) = 0 = c_{d+1, e}(2d+2)$, the $(2, 2)$-entry of $W_x'$ is
\begin{align*}
    & \left( \sum_{i = 1}^d c_{d,e}(2i-1) m_{2d+1} x^{2i} \right) + 1 + \sum_{i = 1}^{d} c_{d,e}(2i) x^{2i} \\
    = & 1 + \left( \sum_{i = 1}^{d} ( c_{d, e}(2i) + c_{d, e}(2i-1) m_{2d+1} ) x^{2i} \right) + c_{d+1, e}(2d+2) x^{2d+2} \\
    = & 1 + \sum_{i=1}^{d+1}c_{d+1,e}(2i) x^{2i}
\end{align*}
So the claim holds for all $d \geq 1$.

For each $k$, the index set $J_d^k$ can be decomposed into $J_{d,o}^k$ and $J_{d, e}^k$.
So the equation $$c_{d,o}(k) + c_{d,e}(k) = \sum_{(j_1,\dots,j_k)\in J_d^k}m_{j_1} \dots m_{j_k} = c_d(k)$$ holds for each $k$.
Hence, the trace of $W_x$ is $2 + \sum_i c_d(2i)x^{2i}$.
\end{proof}

The previous lemma furnishes an expression for traces about a specific class of matrix products. Notably, the condition delineated within this lemma comprehensively encompasses all instances of matrices in their standard forms.

\begin{rmk}
By \reflem{trForm}, if $m_1, \dots, m_{2d}$ are positive, then the trace of $W_1(m_1, \dots, m_{2d})$ is greater than two. Therefore, $W_1(m_1, \dots, m_{2d})$ is a hyperbolic element of $\PSL{\RR}$.
\end{rmk}

\subsection{Comparison of traces and anti-traces}
We previously established a standard form for a hyperbolic element in $\GL(\ZZ)$, represented as $$W_1(m_1, m_2, \dots, m_{2d})$$ with positive integers $m_1, m_2, \dots, m_{2d}$. This form facilitates entry comparisons, as shown in the subsequent proposition.

\begin{prop}\label{Prop:increasingEntry}
For every $d\in \NN$ and $m_1,m_2,\cdots, m_{2d}\in \NN$, the following holds for the matrix $W=W_1(m_1,m_2,\cdots, m_{2d})$:
\begin{itemize}
    \item $W_{1,1}>W_{1,2}\geq W_{2,2}>0$, and
    \item $W_{1,1}>W_{2,1}\geq W_{2,2}>0.$
\end{itemize}
\end{prop}
\begin{proof}
First, we consider $$W^1=W_1(m_1,m_2)=
\begin{pmatrix}
1+m_1m_2&m_1\\
m_2&1
\end{pmatrix}.$$
Observe that 
$$W^1_{1,1}-W^1_{1,2}=1+m_1m_2-m_1=1+m_1(m_2-1)\geq 1 >0$$
and 
$$W^1_{1,1}-W^1_{2,1}=1+m_1m_2-m_2=1+m_2(m_1-1)\geq 1 >0$$
as $m_1,m_2\geq 1$.
Therefore, 
\begin{itemize}
    \item $W^1_{1,1}>W^1_{1,2}\geq W^1_{2,2}>0$, and
    \item $W^1_{1,1}>W^1_{2,1}\geq W^1_{2,2}>0.$
\end{itemize}
Now, for $k\in\{1,\cdots, d\}$, we define $W^k=W_1(m_1,\cdots,m_{2k})$. Assume that, for some $n\in \{1,\cdots, d-1\}$, $$W^n=
\begin{pmatrix}
a&b\\
c&d
\end{pmatrix}
$$
satisfies that $a>b\geq d>0$, and $a>c\geq d>0.$
Note that  
\begin{align*}
W^{n+1}&=W^nW_1(m_{2n+1},m_{2n+2})\\
&=\begin{pmatrix}
a&b\\
c&d
\end{pmatrix}
\begin{pmatrix}
1+m_{2n+1}m_{2n+2}&m_{2n+1}\\
m_{2n+2}&1
\end{pmatrix}\\
&=\begin{pmatrix}
a+am_{2n+1}m_{2n+2}+bm_{2n+2}&am_{2n+1}+b\\
c+cm_{2n+1}m_{2n+2}+dm_{2n+2}& cm_{2n+1}+d
\end{pmatrix}.
\end{align*}
Then, 
\begin{align*}
W^{n+1}_{1,1}-W^{n+1}_{1,2}&=a+am_{2n+1}m_{2n+2}+bm_{2n+2}-am_{2n+1}-b\\
&=a+am_{2n+1}(m_{2n+2}-1)+b(m_{2n+2}-1)\\
&>0
\end{align*}
as  $a>0$ and $m_{2n+2}\geq 1$,
and 
\begin{align*}
W^{n+1}_{1,1}-W^{n+1}_{2,1}&=a+am_{2n+1}m_{2n+2}+bm_{2n+2}-c-cm_{2n+1}m_{2n+2}-dm_{2n+2}\\
&=(a-c)(1+m_{2n+1}m_{2n+2})+(b-d)m_{2n+2}\\
&>0
\end{align*}
as $a>c$ and $b\geq d$.
Similarly,
\begin{align*}
W^{n+1}_{1,2}-W^{n+1}_{2,2}&=am_{2n+1}+b-cm_{2n+1}-d\\
&=(a-c)m_{2n+1}+(b-d)\\
&>0
\end{align*}
as $a>c$ and $b\geq d$
and 
\begin{align*}
W^{n+1}_{2,1}-W^{n+1}_{2,2}&=c+cm_{2n+1}m_{2n+2}+dm_{2n+2}-cm_{2n+1}-d\\
&=c+cm_{2n+1}(m_{2n+2}-1)+d(m_{2n+2}-1)\\
&>0
\end{align*}
as  $c>0$ and $m_{2n+2}\geq 1$.
Therefore, we can see that 
\begin{itemize}
    \item $W^{n+1}_{1,1}>W^{n+1}_{1,2}\geq W^{n+1}_{2,2}>0$, and
    \item $W^{n+1}_{1,1}>W^{n+1}_{2,1}\geq W^{n+1}_{2,2}>0.$
\end{itemize}
Thus, the desired result follows. 
\end{proof}

As a result of the aforementioned, we can compare the trace and anti-trace of a standard form.

\begin{lem}\label{Lem:comparisionEven}
For every $d\in \NN$ and $m_1,m_2,\cdots, m_{2d}\in \NN$,
\[
\frac{\tr^*(W_1(m_1,m_2,\cdots,m_{2d}))}{\tr(W_1(m_1,m_2,\cdots,m_{2d}))}<1.
\]
\end{lem}
\begin{proof}
First, we consider $$W^1=W_1(m_1,m_2)=
\begin{pmatrix}
1+m_1m_2&m_1\\
m_2&1
\end{pmatrix}.$$
Then, 
\begin{align*}
\tr(W^1)-\tr^*(W^1)&=2+m_1m_2-m_1-m_2\\
&=1+(m_1-1)(m_2-1)>0
\end{align*}
as $m_1,m_2\geq 1$. As $\tr(W^1)>0$,
this implies 
\[
\frac{\tr^*(W^1)}{\tr(W^1)}<1.
\]

Now, for $k\in\{1,\cdots, d\}$, we define $W^k=W_1(m_1,\cdots,m_{2k})$. Assume that there is $n\in \{1,\cdots, d-1\}$ such that $$W^n=
\begin{pmatrix}
a&b\\
c&d
\end{pmatrix}
$$
satisfies that 
\[
\frac{\tr^*(W^n)}{\tr(W^n)}=\frac{b+c}{a+d}<1
\]
Note that  
\begin{align*}
W^{n+1}&=W^nW_1(m_{2n+1},m_{2n+2})\\
&=\begin{pmatrix}
a&b\\
c&d
\end{pmatrix}
\begin{pmatrix}
1+m_{2n+1}m_{2n+2}&m_{2n+1}\\
m_{2n+2}&1
\end{pmatrix}\\
&=\begin{pmatrix}
a+am_{2n+1}m_{2n+2}+bm_{2n+2}&am_{2n+1}+b\\
c+cm_{2n+1}m_{2n+2}+dm_{2n+2}& cm_{2n+1}+d
\end{pmatrix}.
\end{align*}
Then,
\begin{align*}
\frac{\tr^*(W^{n+1})}{\tr(W^{n+1})}&=\frac{c+cm_{2n+1}m_{2n+2}+dm_{2n+2}+am_{2n+1}+b}{a+am_{2n+1}m_{2n+2}+bm_{2n+2}+cm_{2n+1}+d}
\\
&= \frac{(cm_{2n+1}m_{2n+2}+am_{2n+1})+dm_{2n+2}+b+c}{(am_{2n+1}m_{2n+2}+cm_{2n+1})+bm_{2n+2}+a+d}.
\end{align*}
By \refprop{increasingEntry}, $b\geq d$ and, by assumption, $a+d>b+c$. Moreover, 
\begin{align*}
(am_{2n+1}m_{2n+2}+cm_{2n+1})-&(cm_{2n+1}m_{2n+2}+am_{2n+1})\\
&=(a-c)m_{2n+1}m_{2n+2}-(a-c)m_{2n+1}\\
&=(a-c)m_{2n+1}(m_{2n+2}-1)\\
&\geq 0.
\end{align*}
Therefore, we can conclude that 
\[
\frac{\tr^*(W^{n+1})}{\tr(W^{n+1})}<1.
\]
Thus, the desired result follows.
\end{proof}
For every odd number $k\in\NN$ and for $m_1,m_2,\cdots,m_k\in \ZZ$, we also write $$W_1(m_1,m_2,\cdots,m_k)=A_1^{m_1}B_1^{m_2}\cdots A_1^{m_{k-2}}B_1^{m_{k-1}}A_1^{m_k}$$
where $W_1(m_1)=A_1^{m_1}$.
Then, we can state similar inequality for odd numbers of products as follow.
\begin{lem}\label{Lem:comparisionOdd}
For every $d\in \NN$ and $m_1,m_2,\cdots, m_{2d+1}\in \NN$,
\[
\frac{\tr(W_1(m_1,m_2,\cdots,m_{2d},m_{2d+1}))}{\tr^*(W_1(m_1,m_2,\cdots,m_{2d},m_{2d+1}))}<1.
\]
\end{lem}
\begin{proof}
For simplicity, we write $$W_{2d}=W_1(m_1,m_2,\cdots,m_{2d})=\begin{pmatrix}
    a&b\\c&d
\end{pmatrix}$$ and 
$$W_{2d+1}=W_1(m_1,m_2,\cdots,m_{2d+1}).$$
By \refprop{increasingEntry} and \reflem{comparisionEven}, we have that
\begin{itemize}
    \item $a>b\geq d$, 
    \item $a>c\geq d$, and
    \item $b+c<a+d.$
\end{itemize}
Note that 
\[
W_{2d+1}=W_{2d}A^{m_{2d+1}}
=\begin{pmatrix}
    a&b\\c&d
\end{pmatrix}\begin{pmatrix}
    1&m_{2d+1}\\0&1
\end{pmatrix}=\begin{pmatrix}
a&am_{2d+1}+b\\c&cm_{2d+1}+d
\end{pmatrix}.
\]
Then, 
\begin{align*}
\tr^*(W_{2d+1})-\tr(W_{2d+1})&=(am_{2d+1}+b+c)-(cm_{2d+1}+a+d)\\
&=(a-c)m_{2d+1}+\{(b+c)-(a+d)\}.
\end{align*}
Since $b+c<a+d$ and $a>c$, we can conclude that 
$$\tr^*(W_{2d+1})-\tr(W_{2d+1})>0 $$
and so 
$$\frac{\tr(W_{2d+1})}{\tr^
*(W_{2d+1})}<1. $$
Therefore, the statement holds.
\end{proof}

\section{On the genus two handlebody} \label{Sec:GenusTwoHandlebody}
In this section, we prove the main theorem, \refthm{almostMinimal}. 
First, in \refsec{lifting},  
We introduce and discuss a homomorphic section $I$  of $A$, namely, $I$ is a homomorphism from $\Out(F_2)$ to $\cH_2$  such that $A \circ I=\id_{\Out(F_2)}$. This homomorphic section has been introduced in \cite[Example~5.8]{Hensel20}, in terms of interval bundle maps. 
However, we need to modify the definition for orientation reversing homeomorphisms. 

In fact, for each fully irreducible element $\varphi$ in $\Out(F_2)$, we will see that $I(\varphi)$ is a reducible element in $A^{-1}(\varphi)$. To construct a pseudo-Anosov map in $A^{-1}(\varphi)$, we take a meridian $\alpha$ so that we can apply the \emph{Penner construction} to the composition of $I(\varphi)$ and $T_\alpha$. To do this, in the end of \refsec{lifting}, we  take a suitable collection of curves. Then, in \refsec{Penner}, we construct a pseudo-Anosov by applying the Penner construction.

In \refsec{largestEigen}, we  estimate the eigenvalue of the Penner's linear representation to compute the stretch factors of the  pseudo-Anosov maps, constructed in \refsec{Penner}.

Finally, in \refsec{main}, we will complete the proof of  \refthm{almostMinimal}.

\subsection{A lifting of $\Out(F_2)$ to $\cH_2$}\label{Sec:lifting}

For each $v=(x_1,x_2,\dots, x_n)\in \RR^n$ and $r>0$, we denote the set 
$$\{(a_1,a_2,\cdots,a_n)\in \RR^n:|a_i-x_i|< r/2 \text{ for all }i  \}$$ 
by $C_r^n(v)$ and call it the \emph{$n$-cude} centered at $v$ with width $r$. For two subsets $A$ and $B$ of $\RR^n$, we write $$A+B=\{a+b:a\in A \text{ and 
}b\in B\}$$ for the \emph{sum} of $A$ and $B$.

Let $q$ be the natural projection map from $\RR^2$ to $\RR^2/\ZZ^2$. Set $$\cC=C_{1/2}^2(0,0)+(1/2,1/2)+\ZZ^2.$$ Then, the image $q(\RR^2\setminus \cC)$ is homeomorphic to the torus $S_1^1$ with one boundary component. Hence, we can identify $q(\RR^2\setminus \cC)$ with $S_1^1$. Here, the boundary component $b$ of $S_1^1$ is $q(\partial C_{1/2}(1/2,1/2))$. Also, we consider $b$ as a punture in the following sense: in usual sense,  a mapping class of a compact surface with boundary is an isotopy class of a boundary fixing homeomoprhism. However,  for our purpose, we consider a mapping class of a compact surface with boundary as  a isotopy class of a boundary preserving homeomoprhism. Hence, we do not distinguish $\Mod^\pm(S_{1,1})$ with  $\Mod^\pm(S_1^1)$. 

Now, we consider the inclusion map $\tilde{j}:\RR^2\setminus \cC \to \RR^2\setminus (1/2,1/2)+\ZZ^2.$ Then, there is the inclusion map $j: S_1^1\to S_{1,1}$ such that $q\circ \tilde{j}=j\circ q.$ Then, $j$ induces an isomorphism $j_*: \pi_1(S_1^1)\to \pi_1(S_{1,1})$ and as $\{x_1,x_2\}$ in \refsec{OutToMod} is a free generating set of  $\pi_1(S_{1,1})$, we may set that $\{x_1,x_2\}$ is a free generating set of $\pi_1(S_1^1)$ via $j_*$. Under this identification, we have that $$\Mod^\pm(S_1^1)=\Mod^{\pm}(S_{1,1})=\Out(F_2)=\GL(\ZZ)$$

We define simple closed curves $\alpha$ and $\beta$ in $S_1^1$ by 
$$\alpha(t)=q(t,0)\text{ and }\beta(t)=q(0,t)$$
for $t\in [0,1].$
Then, $\pi_1(S_1^1,q(0,0))$ is a rank two free group generated by $\alpha$ and $\beta.$ Also, by the identification via $j_*$, $x_1$ and $x_2$ are corresponding to $\alpha$ and $\beta$, respectively.
Moreover, the Dehn twists $T_a$ and $T_b$ in $\Mod^\pm(S_1^1)$ about $a=\im{\alpha}$ and $b=\im{\beta}$, respectively, are corresponded to $A_1=\begin{pmatrix}
    1&1\\0&1
\end{pmatrix}$ and $
B_{-1}=\begin{pmatrix}
    1&0\\-1&1
\end{pmatrix}
$, respectively, in $\GL(\ZZ)$.

We consider $\widetilde{V_2}=(\RR^2\setminus \cC)\times J$ where $J=[-1/2,1/2]$ and define $\pi: \widetilde{V_2}\to S_1^1\times J$ by $(x,t)\mapsto (q(x),t)$. Note that $\pi$ is a covering projection and that $S_1^1\times J$ is a genus two handlebody $V_2$. We define an inclusion $k:S_1^1 \to S_1^1\times J$ by $k(x)=(x,0)$. Note that the homomorphism $k_* : \pi_1(S_1^1,q(0)) \to \pi_1(V_2, \pi(0))$ induced from $k$ is an isomorphism. We also identify $\pi_1(S_1^1)$ with $\pi_1(V_2)$ so that $\pi_1(V_2)$ is generated by $\{x_1,x_2\}$ via $k_* \circ j_*^{-1}$

We define an injective homomorphism $I^\circ$ from $\Homeo(S_1^1)$ to $\Homeo(V_2)$ as follows:
for $h\in \Homeo(S_1^1)$ and for $(x,t)\in V_2$,
\begin{align*}
    I^\circ(h)(x,t)=
    \begin{cases}
     (h(x),t) &\text{, if}~h\in \Homeop(S_1^1),\\
     (h(x),-t) &\text{, otherwise.}
    \end{cases}
\end{align*}
Then, $I^\circ$ induces the injective homomorphism $I:\Mod^\pm(S_1^1)\to \cH_2$ so that $I([h])=[I^\circ(h)]$. Here, $[\cdot]$ denotes the isotopy class. Then, $I$ is a homomorphic section of the surjection $A:\cH_2 \twoheadrightarrow \Out(F_2)$. To see this, $A\circ I: \Mod^\pm(S_1^1)\to  \Out(F_2)$ is an isomorphism by \cite[Theorem~8.8]{FarbMargalit12}. Under the identifications in \refsec{OutToGL} with the above, $A\circ I$ is the identity on $ \Out(F_2)$. Therefore, we can say that $I$ is a section of $A$ which is a homomorphism from $\Out(F_2)$ to $\cH_2$.  

Then, $I(T_a)$ is the annulus twist $T_{a\times \frac{1}{2}}T_{a\times -\frac{1}{2}}^{-1}$ about the annulus $a\times J$ in $\cH_2$. Also, $I(T_b)$ is the annulus twist $T_{b\times \frac{1}{2}}T_{b\times -\frac{1}{2}}^{-1}$ about the annulus $b\times J$ in $\cH_2$. 
As $\Out(F_2)=\GL(\ZZ)$ in \refsec{OutToGL}, we can say that 
$$A(T_{a\times \frac{1}{2}}T_{a\times -\frac{1}{2}}^{-1})= \begin{pmatrix}
1&1\\0&1    
\end{pmatrix}
\text{ and }
A(T_{b\times \frac{1}{2}}T_{b\times -\frac{1}{2}}^{-1})=\begin{pmatrix}
1&0\\-1&1    
\end{pmatrix}.
$$

Then, we consider the matrices 
$$\tilde{e}=\begin{pmatrix}
0&1\\
1&0\\
\end{pmatrix}\text{ and }\tilde{E}=\begin{pmatrix}
0&1&0\\
1&0&0\\
0&0&-1
\end{pmatrix}.$$
We think of $\tilde{e}$ and $\tilde{E}$  as linear transformations  on  $\RR^2$ and $\RR^3$, respectively. Note that $\tilde{e}(\RR^2\setminus \cC)= \RR^2\setminus \cC$ and  $\tilde{E}(\tilde{V_2})=\tilde{V_2}$. Then, 
there are  unique homeomorphisms $e$ and 
$E$ on $S_1^1$ and $V_2$, respectively, such that $e\circ q = q \circ \tilde{e}$ and $E\circ \pi = \pi \circ \tilde{E}$. Observe that $I([e])=[E]$. Like the above, we can say that $A([E])=\tilde{e}$. Also, as $\tilde{E}(t,0,1/2)=(0,t,-1/2)$ and $\tilde{E}(0,t,1/2)=(t,0,-1/2)$, $E(a\times \frac{1}{2})=b\times -\frac{1}{2}$  and $E(b\times \frac{1}{2})=a\times -\frac{1}{2}$.

Similarly, we consider the matrices 
$$\tilde{r}=\begin{pmatrix}
-1&0\\
0&-1\\
\end{pmatrix}\text{ and }\tilde{R}=\begin{pmatrix}
-1&0&0\\
0&-1&0\\
0&0&1
\end{pmatrix}.$$
We think of $\tilde{r}$ and $\tilde{R}$  as linear transformations  on  $\RR^2$ and $\RR^3$, respectively. Note that $\tilde{r}(\RR^2\setminus \cC)= \RR^2\setminus \cC$ and  $\tilde{R}(\tilde{V_2})=\tilde{V_2}$. Then, 
there are  unique homeomorphisms $r$ and 
$R$ on $S_1^1$ and $V_2$, respectively, such that $r\circ q = q \circ \tilde{e}$ and $R\circ \pi = \pi \circ \tilde{R}$. Observe that $I([r])=[R]$. Like the above, we can say that $A([R])=\tilde{r}$. Also, as $\tilde{R}(t,0,\pm1/2)=(-t,0,\pm1/2)$ and $\tilde{R}(0,t,\pm1/2)=(0,-t,\pm1/2)$, $R(a\times \pm\frac{1}{2})=a\times \pm \frac{1}{2}$  and $R(b\times \pm \frac{1}{2})=b\times \pm\frac{1}{2}$ with respect to sign. Moreover, $E\circ R=R\circ E$.

Finally, we consider a simple closed curve   $\gamma$ in $\partial V_2$ defined by  

\begin{align*}
\gamma(t)=
\begin{cases}
\pi(-\frac{1}{4},4t-\frac{1}{4},\frac{1}{2})&\text{ for }  0\leq t \leq \frac{1}{8}\\
\pi(-\frac{1}{4},\frac{1}{4},-8t+\frac{3}{2})&\text{ for } \frac{1}{8}\leq t \leq \frac{2}{8}\\
\pi(4t-\frac{5}{4},\frac{1}{4},-\frac{1}{2})&\text{ for }  \frac{2}{8}\leq t \leq \frac{3}{8}\\
\pi(\frac{1}{4},\frac{1}{4},8t-\frac{7}{2})&\text{ for }  \frac{3}{8}\leq t \leq \frac{4}{8}\\
\pi(\frac{1}{4},-4t+\frac{9}{4},\frac{1}{2 })&\text{ for }  \frac{4}{8}\leq t \leq \frac{5}{8}\\
\pi(\frac{1}{4},-\frac{1}{4}, -8t+\frac{11}{2} )&\text{ for }  \frac{5}{8}\leq t \leq \frac{6}{8}\\
\pi(-4t+\frac{13}{4}, -\frac{1}{4},-\frac{1}{2})&\text{ for }  \frac{6}{8}\leq t \leq \frac{7}{8}\\
\pi(-\frac{1}{4},-\frac{1}{4},8t-\frac{15}{2})&\text{ for }  \frac{7}{8}\leq t \leq 1.
\end{cases}
\end{align*}
Observe that $c=\im(\gamma)$ is a meridian of $\partial V_2$. Moreover, $E(c)=c$ and $R(c)=c$. 

For the simplicity, we write 
$$\delta_1=a\times \frac{1}{2}, \ \delta_2=b\times \frac{1}{2}, \ \delta_3=a\times -\frac{1}{2}, \ \delta_4=b\times -\frac{1}{2},\text{ and }\delta_5=c.$$

\subsection{The Penner construction}\label{Sec:Penner}
Note that by \refsec{trAndStr} and by \cite[Example~1.4]{BestvinaHandel92}, $\varphi$ is either a hyperbolic element or glide reflection in $\PGL{\ZZ}$ if and only if $\varphi$ is a fully irreducible in $\Out(F_2)=\GL(\ZZ)$.  

In this section, for each fully irreducible element $\varphi$ in $\Out(F_2)$, we construct a pseudo-Anosov element $\varphi^\star$ in $\cH_2$ such that $A(\varphi^\star)=\varphi$. 
To this end, we apply the Penner construction of a pseudo-Anosov mapping class \cite{Penner88}.
From now on, we follow the notations in \cite{Penner88}. See \cite[194~page]{Penner88}. 

Let $\sC^+=\{\delta_1,\delta_4\}$ and $\sD^-=\{\delta_2,\delta_3,\delta_5\}$. Also, $\sE=\{\delta_k\}_{k=1}^5=\sC^+ \cup \sD^-$ and $T$ denote the train track obtained in the proof of \cite[Theorem~3.1]{Penner88} with branches $\{b_i\}$. Then, $S(\sC^+,\sD^-)$ denotes the semigroup defined as
$$S(\sC^+,\sD^-)=\langle e\in \sE:[e,e']=1 \text{ if } e\cap e'=\emptyset \rangle.$$ For each $k$, there is a unique curve $t(\delta_k)$ in $T$ isopotic to $\delta_i$ and a measure $\mu_k$ on $T$ defined by 
\begin{align*}
\mu_k(b_i)=
\begin{cases}
1 &,\ b_i\subset t(e_k)\\
0 &,\ \text{otherwise.}\\
\end{cases}
\end{align*}
Then, $H$ is the convex hull of $\{\mu_k\}_{k=1}^5$ and the zero measure in the cone of measures on $T$. Note that $\{\mu_k\}_{k=1}^5$ is a positive basis for $H$.

The intersection matrix $\cA$ is defined by $\cA_{mn}=\card(\delta_m\cap \delta_n)$ and for each $k\in \{1,2,3,4,5\}$, the matrix $\cA_k$ is defined by 
\begin{align*}
   (\cA_k)_{mn}= \begin{cases}
   A_{mn}&,\ m=k\\
 0&,\ \text{otherwise.}\\
   \end{cases}
\end{align*}

\begin{thm}[Theorem 3.4 in \cite{Penner88}]\label{Thm:MRep}
The action of $S(\sC^+,\sD^-)$ on $H$  admits a faithful representation $\rho$ with respect to the basis $\{\mu_k\}$ as a semigroup of invertible (over $\ZZ$) positive matrice so that  
$\rho(\delta_k)=I_5+\cA_k$. Here, $I_5$ is the $5\times5$-identity matrix. 
\end{thm}

There is a semigroup homomorphism $\zeta : S(\sC^+,\sD^-)\to \cH_2$ such that for each $\delta\in \sE$,
\begin{align*}
  \zeta(\delta)=
  \begin{cases}
      T_\delta&,  \delta\in \sC^+\\
       T_\delta^{-1}&, \delta \in \sD^-.\\
  \end{cases}
\end{align*}

\begin{cor}[Corollary 3.5
in \cite{Penner88}]\label{Cor:PennCstr}
The map $\zeta$ is injective.
Moreover, for a word $w$ in $S(\sC^+,\sD^-)$, $\zeta(w)$ is pseudo-Anosov if each element in $\sE$ occurs at least once in $w$. In this case, the stretch factor of $\zeta(w)$ is equal to the Perron-Frobenius eigenvalue of $\rho(w)$.    
\end{cor}

Now, we come back to the construction problem.

\subsubsection{The case for hyperbolic isometries}\label{Sec:hyperbolic}
Let $M$ be a hyperbolic isometry in $\GL(\ZZ)$. Then, we can take the standard form $S(M)$ of $M$ in $\cO_{\GL(\ZZ)}(M)$ and so 
$$S(M)=\pm \begin{pmatrix}
1&1\\0&1    
\end{pmatrix}^{a_1}\begin{pmatrix}
1&0\\1&1    
\end{pmatrix}^{a_2}\cdots \begin{pmatrix}
1&1\\0&1    
\end{pmatrix}^{a_{d-1}}\begin{pmatrix}
1&0\\1&1    
\end{pmatrix}^{a_d} $$
for some positive integer $a_i$ where $d=d(M)$. If there is a pseudo-Anosov $S(M)^\star$ in $\cH_2$ such that $A(S(M)^\star)=S(M)$, then $M^\star=\psi S(M)^\star \psi^{-1}$ for some $\psi\in \cH_2$ as $S(M)\in \cO_{\GL(\ZZ)}(M)$ and $A$ is surjective.  It is enough to show that there is $S(M)^\star$. 

\textbf{Case 1} : The sign of $\tr(M)$ is positive, that is, $$S(M)=+ \begin{pmatrix}
1&1\\0&1    
\end{pmatrix}^{a_1}\begin{pmatrix}
1&0\\1&1    
\end{pmatrix}^{a_2}\cdots \begin{pmatrix}
1&1\\0&1    
\end{pmatrix}^{a_{d-1}}\begin{pmatrix}
1&0\\1&1    
\end{pmatrix}^{a_d}.$$ 

We take a word $w(S(M))=\delta_5\prod_{i=1}^{d(M)/2}(\delta_1\delta_3)^{a_{2i-1}}(\delta_2\delta_4)^{a_{2i}}$ in $S(\sC^+,\sD^-)$. Then, 
$$ \zeta(w(S(M))) =  T_{\delta_5}^{-1}\prod_{i=1}^{d(M)/2}(T_{\delta_1}T_{\delta_3}^{-1})^{a_{2i-1}}(T_{\delta_2}^{-1}T_{\delta_4})^{a_{2i}}.$$
By \refcor{PennCstr}, $\zeta(w(S(M)))$ is a pseudo-Anosov mapping class, the stretch factor of which is equal to the Perron-Frobenius eigenvalue of $\rho(w(S(M)))$.
Since $A(T_{\delta_5}^{-1})=I_2$, $A(T_{\delta_1}T_{\delta_3}^{-1})=A_1$, and $A(T_{\delta_2}^{-1}T_{\delta_4})=B_1$, 
$A(\zeta(w(S(M))))=S(M)$
where $I_2$ is the identity in $\GL(\ZZ)$.
Therefore, we set $S(M)^\star=\zeta(w(S(M)))$.

\textbf{Case 2} : The sign of $\tr(M)$ is negative, that is, $$S(M)=- \begin{pmatrix}
1&1\\0&1    
\end{pmatrix}^{a_1}\begin{pmatrix}
1&0\\1&1    
\end{pmatrix}^{a_2}\cdots \begin{pmatrix}
1&1\\0&1    
\end{pmatrix}^{a_{d-1}}\begin{pmatrix}
1&0\\1&1    
\end{pmatrix}^{a_d}.$$
We take a word $w(S(M))=\delta_5\prod_{i=1}^{d(M)/2}(\delta_1\delta_3)^{a_{2i-1}}(\delta_2\delta_4)^{a_{2i}}$ in $S(\sC^+,\sD^-)$. Then, $$ \zeta(w(S(M))) =  T_{\delta_5}^{-1}\prod_{i=1}^{d(M)/2}(T_{\delta_1}T_{\delta_3}^{-1})^{a_{2i-1}}(T_{\delta_2}^{-1}T_{\delta_4})^{a_{2i}}.$$
As $A([R])=\tilde{r}$ in \refsec{lifting}, 
we have $$A(\zeta(w(S(M)))[R])=S(M).$$
To see that $\zeta(w(S(M)))[R]$ is pseudo-Anosov, we consider the square  $$\zeta(w(S(M)))[R]\circ \zeta(w(S(M)))[R].$$ Note that $[R]=[R]^{-1}$ and $[R]\circ T_\delta \circ[R]=T_\delta$ for all $\delta\in \sE$. Therefore, 
$$\zeta(w(S(M)))[R]\circ \zeta(w(S(M)))[R]=\zeta(w(S(M)))\circ \zeta(w(S(M))).$$
By \refcor{PennCstr}, $\zeta(w(S(M)))\circ \zeta(w(S(M)))$ is pseudo-Anosov. This implies that $\zeta(w(S(M)))[R]$ is also pseudo-Anosov. Thus, we set $S(M)^\star=\zeta(w(S(M)))[R]$.

\subsubsection{The case for glide reflections}\label{Sec:glideRef}
Let $M$ be a glide reflection in $\GL(\ZZ)$. Then, we can take the standard form $S(M)$ of $M$ in $\cO_{\GL(\ZZ)}(M)$ and so 
$$S(M)=\pm \begin{pmatrix}
1&1\\0&1    
\end{pmatrix}^{a_1}\begin{pmatrix}
1&0\\1&1    
\end{pmatrix}^{a_2}\cdots \begin{pmatrix}
1&0\\1&1    
\end{pmatrix}^{a_{d-1}} \begin{pmatrix}
1&1\\0&1    
\end{pmatrix}^{a_{d}} \begin{pmatrix}
0&1\\0&1    
\end{pmatrix}
$$
for some positive integer $a_i$ where $d=d(M)$. Note that $d$ can be $1$. Again,  it is enough to show that there is $S(M)^\star$. 

\textbf{Case 1} : The sign of $\tr(M)$ is positive, that is, $$S(M)=+ \begin{pmatrix}
1&1\\0&1    
\end{pmatrix}^{a_1}\begin{pmatrix}
1&0\\1&1    
\end{pmatrix}^{a_2}\cdots \begin{pmatrix}
1&0\\1&1    
\end{pmatrix}^{a_{d-1}} \begin{pmatrix}
1&1\\0&1    
\end{pmatrix}^{a_{d}} \begin{pmatrix}
0&1\\1&0    
\end{pmatrix}
$$ 

We take a word $w(S(M))=\delta_5 \left[\prod_{i=1}^{(d(M)-1)/2}(\delta_1\delta_3)^{a_{2i-1}}(\delta_2\delta_4)^{a_{2i}}\right] (\delta_1\delta_3)^{a_d}$ in $S(\sC^+,\sD^-)$. 
Then, 
$$ \zeta(w(S(M))) =  T_{\delta_5}^{-1}\left[\prod_{i=1}^{(d(M)-1)/2}(T_{\delta_1}T_{\delta_3}^{-1})^{a_{2i-1}}(T_{\delta_2}^{-1}T_{\delta_4})^{a_{2i}}\right](T_{\delta_1}T_{\delta_3}^{-1})^{a_d}.$$

As $A([E])=\tilde{e}$ in \refsec{lifting}, we have 
$$A(\zeta(w(S(M)))[E])=S(M).$$

To see that $\zeta(w(S(M)))[E]$ is pseudo-Anosov, we observe that the square $\zeta(w(S(M)))[E]\circ \zeta(w(S(M)))[E]$ is pseudo-Anosov. Note that $[E]=[E]^{-1}$, $[E]\circ T_{\delta_1}\circ[E]=T_{\delta_4}$, $[E]\circ T_{\delta_2}\circ[E]=T_{\delta_3}$ and $[E]\circ T_{\delta_5}\circ[E]=T_{\delta_5}$ . Hence, 

\begin{align*}
\zeta(w(S(M)))[E]&\circ \zeta(w(S(M)))[E]\\
&=  T_{\delta_5}^{-1}\left[\prod_{i=1}^{(d(M)-1)/2}(T_{\delta_1}T_{\delta_3}^{-1})^{a_{2i-1}}(T_{\delta_2}^{-1}T_{\delta_4})^{a_{2i}}\right](T_{\delta_1}T_{\delta_3}^{-1})^{a_d}\\
&\quad \circ T_{\delta_5}^{-1}\left[\prod_{i=1}^{(d(M)-1)/2}(T_{\delta_4}T_{\delta_2}^{-1})^{a_{2i-1}}(T_{\delta_3}^{-1}T_{\delta_1})^{a_{2i}}\right](T_{\delta_4}T_{\delta_2}^{-1})^{a_d}\\
&=  T_{\delta_5}^{-1}\left[\prod_{i=1}^{(d(M)-1)/2}(T_{\delta_1}T_{\delta_3}^{-1})^{a_{2i-1}}(T_{\delta_2}^{-1}T_{\delta_4})^{a_{2i}}\right](T_{\delta_1}T_{\delta_3}^{-1})^{a_d}\\
&\quad \circ T_{\delta_5}^{-1}\left[\prod_{i=1}^{(d(M)-1)/2}(T_{\delta_2}^{-1}T_{\delta_4})^{a_{2i-1}}(T_{\delta_1}T_{\delta_3}^{-1})^{a_{2i}}\right](T_{\delta_2}^{-1}T_{\delta_4})^{a_d}.
\end{align*}
By \refcor{PennCstr}, we can see that $\zeta(w(S(M)))[E]\circ \zeta(w(S(M)))[E]$ is pseudo-Anosov and so $\zeta(w(S(M)))[E]$ is pseudo-Anosov. Therefore, we set $S(M)^\star=\zeta(w(S(M)))[E]$. 

\textbf{Case 2} : The sign of $\tr(M)$ is negative, that is, $$S(M)=- \begin{pmatrix}
1&1\\0&1    
\end{pmatrix}^{a_1}\begin{pmatrix}
1&0\\1&1    
\end{pmatrix}^{a_2}\cdots \begin{pmatrix}
1&0\\1&1    
\end{pmatrix}^{a_{d-1}} \begin{pmatrix}
1&1\\0&1    
\end{pmatrix}^{a_{d}} \begin{pmatrix}
0&1\\1&0    
\end{pmatrix}
$$ 
We take a word $w(S(M))=\delta_5 \left[\prod_{i=1}^{(d(M)-1)/2}(\delta_1\delta_3)^{a_{2i-1}}(\delta_2\delta_4)^{a_{2i}}\right] (\delta_1\delta_3)^{a_d}$ in $S(\sC^+,\sD^-)$. 
Then, 
$$ \zeta(w(S(M))) =  T_{\delta_5}^{-1}\left[\prod_{i=1}^{(d(M)-1)/2}(T_{\delta_1}T_{\delta_3}^{-1})^{a_{2i-1}}(T_{\delta_2}^{-1}T_{\delta_4})^{a_{2i}}\right](T_{\delta_1}T_{\delta_3}^{-1})^{a_d}.$$

As $A([E])=\tilde{e}$ and $A([R])=\tilde{r}$ in \refsec{lifting}, 
we have $$A(\zeta(w(S(M)))[E][R])=S(M).$$
Now, we show that $\zeta(w(S(M)))[E][R]$ is pseudo-Anosov. 
Note that $[R][E]=[E][R]$. Hence,
\begin{align*}
\zeta(w(S(M)))[E][R]&\circ \zeta(w(S(M)))[E][R]\\
&= \zeta(w(S(M)))[E]\circ \zeta(w(S(M)))[E]    
\end{align*}
Then, as in the previous case, we can see that $\zeta(w(S(M)))[E]\circ \zeta(w(S(M)))[E]$ is pseudo-Anosov and so $\zeta(w(S(M)))[E][R]$ is also pseudo-Anosov. Therefore, we set $S(M)^\star=\zeta(w(S(M)))[E][R]$. 
\subsection{Largest eigenvalues}\label{Sec:largestEigen}

Our main theorem comes from the comparison of the stretch factor of an fully irreducible outer automorphism $\varphi$ with the stretch factor of $\varphi^\star$. The stretch factor of $\varphi$ can be evaluated from $\tr(\varphi)$. In this section, we do an elementary linear algebra to estimate the stretch factor of $\varphi^\star$ by using \refcor{PennCstr}. To this end, we study the representation $\rho$ in \refthm{MRep}.  

\begin{lem}\label{Lem:eigenValues}
    Write $$\fA=\begin{pmatrix}
        1 & 1 & 0 & 0 & 2 \\
        0 & 1 & 0 & 0 & 0 \\
        0 & 0 & 1 & 1 & 0 \\
        0 & 0 & 0 & 1 & 0 \\
        0 & 0 & 0 & 0 & 1
    \end{pmatrix}, \quad \fB= \begin{pmatrix}
        1 & 0 & 0 & 0 & 0 \\
        1 & 1 & 0 & 0 & 0 \\
        0 & 0 & 1 & 0 & 0 \\
        0 & 0 & 1 & 1 & 2 \\
        0 & 0 & 0 & 0 & 1
    \end{pmatrix}, \quad \fC = \begin{pmatrix}
        1 & 0 & 0 & 0 & 0 \\
        0 & 1 & 0 & 0 & 0 \\
        0 & 0 & 1 & 0 & 0 \\
        0 & 0 & 0 & 1 & 0 \\
        2 & 0 & 0 & 2 & 1
    \end{pmatrix}.$$
    For every $d\in \NN$ and $m_1,\cdots, m_{2d} \in \ZZ$, the product $\fW(m_1,\cdots,m_{2d})=\prod_{i=1}^d \fA^{m_{2i-1}}\fB^{m_{2i}}$ is written by 
    \[
    \begin{pmatrix}
        a & b & 0 & 0 & 2b \\
        c & d & 0 & 0 & 2d-2 \\
        0 & 0 & a & b & 2a-2 \\
        0 & 0 & c & d & 2c \\
        0 & 0 & 0 & 0 & 1
    \end{pmatrix}    
    \]
where \[W_1(m_1,\cdots,m_{2d})=\prod_{i=1}^{d}A_1^{m_{2i-1}}B_1^{m_{2i}}=
\begin{pmatrix}
    a&b\\
    c&d
\end{pmatrix}.
\]
Therefore, $\fC\fW$ can be written by 
\[
    \begin{pmatrix}
        a & b & 0 & 0 & 2b \\
        c & d & 0 & 0 & 2d-2 \\
        0 & 0 & a & b & 2a-2 \\
        0 & 0 & c & d & 2c \\
        2a&2b & 2c&2d & 1+4(b+c)
    \end{pmatrix}.    
\]
    The characteristic polynomial $\det(xI_5-\fC\fW)$ of $\fC\fW$ is $$(x - 1) \left(x^2 - \tr(W_1)x + 1 \right) \left( x^2 - \left( \tr(W_1)+ 4\tr^*(W_1) \right)x + 1 \right).$$
\end{lem}
\begin{proof}
The first statement follows from the following observation.
Let $\fG$ be the set of all $5\times 5$-matrices of the following form:
\[
  \begin{pmatrix}
        a & b & 0 & 0 & 2b \\
        c & d & 0 & 0 & 2d-2 \\
        0 & 0 & a & b & 2a-2 \\
        0 & 0 & c & d & 2c \\
        0 & 0 & 0 & 0 & 1
    \end{pmatrix}.\]
If there are two matrices 
\[
  \begin{pmatrix}
        a & b & 0 & 0 & 2b \\
        c & d & 0 & 0 & 2d-2 \\
        0 & 0 & a & b & 2a-2 \\
        0 & 0 & c & d & 2c \\
        0 & 0 & 0 & 0 & 1
    \end{pmatrix} \text{ and }   \begin{pmatrix}
        e & f & 0 & 0 & 2f \\
        g & h & 0 & 0 & 2h-2 \\
        0 & 0 & e & f & 2e-2 \\
        0 & 0 & g & h & 2g \\
        0 & 0 & 0 & 0 & 1
    \end{pmatrix}, 
\] then 
\begin{align*}
    &\begin{pmatrix}
        a & b & 0 & 0 & 2b \\
        c & d & 0 & 0 & 2d-2 \\
        0 & 0 & a & b & 2a-2 \\
        0 & 0 & c & d & 2c \\
        0 & 0 & 0 & 0 & 1
    \end{pmatrix}\begin{pmatrix}
        e & f & 0 & 0 & 2f \\
        g & h & 0 & 0 & 2h-2 \\
        0 & 0 & e & f & 2e-2 \\
        0 & 0 & g & h & 2g \\
        0 & 0 & 0 & 0 & 1
    \end{pmatrix}\\
    &= \begin{pmatrix}
        ae+bg & af+bh & 0 & 0 & 2(af+bh) \\
        ce+dg & cf+dh & 0 & 0 & 2(cf+dh)-2 \\
        0 & 0 & ae+bg & af+bh & 2(ae+bg)-2 \\
        0 & 0 & ce+dg & cf+dh & 2(ce+dg) \\
        0 & 0 & 0 & 0 & 1
    \end{pmatrix} \\
\end{align*}
This implies that $\fG$ is a semigroup with the matrix multiplication. Therefore, as $\fA,\fB\in \fG$, $\fW\in\fG$ and, again, by the above computation, the first statement follows.
The second statement is obtained by direct computation.
\end{proof}
\begin{rmk}\label{Rmk:MRep}
Note that $\rho(\delta_1\delta_3)=\fA$, $\rho(\delta_2\delta_4)=\fB$, and $\rho(\delta_5)=\fC$ by \refthm{MRep}.
\end{rmk}

\begin{lem}\label{Lem:distEigen}
For any $d\in\NN$ and for $m_1,m_2,\cdots, m_{2d}\in \NN$, $\fC \fW$ has five distinct positive real  eigenvalues, $1$ and the following four:
\[
\frac{t+4s+\sqrt{(t+4s)^2-4}}{2},\
\frac{t+4s-\sqrt{(t+4s)^2-4}}{2},\
\frac{t+\sqrt{t^2-4}}{2}
\text{ and }
\frac{t-\sqrt{t^2-4}}{2}.
\]
where $t$ and $s$ are the trace and anti-trace of $W_1(m_1,m_2,\cdots, m_{2d}),$ respectively. Moreover, 
\[
\frac{t+4s+\sqrt{(t+4s)^2-4}}{2}> \frac{t+\sqrt{t^2-4}}{2}
>1>
\frac{t-\sqrt{t^2-4}}{2}
>\frac{t+4s-\sqrt{(t+4s)^2-4}}{2}.
\]
\end{lem}
\begin{proof}
By \refprop{increasingEntry}, $t$ and $s$ are positive integers bigger than or equal to $2$. Hence, $(t+4s)^2-4$ and $t^2-4$ are non-negative integers. This implies that the eigenvalues are real numbers. 
Also, we have that 
$$t+4s>t\geq2$$
and
$$\sqrt{(t+4s)^2-4}>\sqrt{t^2-4}\geq 0$$ and 
that $$a-\sqrt{a^2-4}=\sqrt{a^2}-\sqrt{a^2-4}>0$$ for all positive real number $a\geq 2$.
This implies that the eigenvalues are positive and  that the first eigenvalue listed in the statement is the largest eigenvalue. Then,  the second statement follows from the facts that the first and second eigenvalues are distinct and the product of them is equal to $1$  and that the third and fourth eigenvalues are distinct and the product of them is equal to $1$. 
\end{proof}
\begin{rmk}\label{Rmk:PowEigen}
Under the assumption in \reflem{distEigen}, any positive $n^{th}$-power of $\fC \fW $ has five distinct real eigenvalues which are the $n^{th}$-powers of the eigenvalues of $\fC \fW $. The largest eigenvalue of $(\fC \fW )^n$ is the $n^{th}$-power of the largest eigenvalue of $\fC \fW $.    
\end{rmk}

\begin{lem}\label{Lem:eigenValues2}
Write 
\[\fE = \begin{pmatrix}
    0 & 0 & 0 & 1 & 0 \\
    0 & 0 & 1 & 0 & 0 \\
    0 & 1 & 0 & 0 & 0 \\
    1 & 0 & 0 & 0 & 0 \\
    0 & 0 & 0 & 0 & 1
\end{pmatrix}.\]
Then,
\[
\fE^{-1}=\fE,\fE\fA\fE=\fB, \fE\fB\fE=\fA, \text{ and }\fE\fC\fE= \fC.
\] and for every $d\in \NN$ and $m_1,\cdots, m_{2d} \in \ZZ$, 
the characteristic polynomial $\det(xI_5-\fC \fW \fE)$ of $\fC \fW(m_1,\cdots,m_{2d} ) \fE$ is
\[(x-1)\{1
-4\tr^*(W)x
-(\tr^*(W)^2+4\tr(W)\tr^*(W)+2)x^2
-4\tr^*(W)x^3
+x^4)\]
where $W=W_1(m_1,\cdots,m_{2d})$. Therefore, the eigenvalues of $\fC\fW\fE$ are $1$ and the following four :

\[
s+\frac{1}{2}\sqrt{5s^2+4ts+4}+\frac{1}{2}\sqrt{9s^2+4ts+4s\sqrt{5s^2+4ts+4}},
\]
\[
s+\frac{1}{2}\sqrt{5s^2+4ts+4}-\frac{1}{2}\sqrt{9s^2+4ts+4s\sqrt{5s^2+4ts+4}},
\]
\[
s-\frac{1}{2}\sqrt{5s^2+4ts+4}+\frac{1}{2}\sqrt{9s^2+4ts-4s\sqrt{5s^2+4ts+4}},\]
and 
\[
s-\frac{1}{2}\sqrt{5s^2+4ts+4}-\frac{1}{2}\sqrt{9s^2+4ts-4s\sqrt{5s^2+4ts+4}}\]
where $t=\tr(W)$ and $s=\tr^*(W)$.
\end{lem}
\begin{proof}

The first statement follows from direct computation. It is enough to show the second statements. In the following computation, we can see that $\det(xI_5-\fC\fW\fE)$ can be factorized.
\begin{align*}
 1
-4sx
-(s^2+4ts+2)x^2
-4sx^3
+x^4&=x^2\{x^2+x^{-2}-4s(x+x^{-1})-(s^2+4ts+2)\}\\
&=x^2\{(x+x^{-1})^2-4s(x+x^{-1})-(s^2+4ts+4)\}\\
&=x^2\{z^2-4sz-(s^2+4ts+4)\}
\end{align*}
where $z=x+x^{-1}$. 
Then, we solve  $z^2-4sz-(s^2+4ts+4)=0$ as an equation of $z$-variable.
\begin{align*}
z&=2s\pm \sqrt{4s^2+(s^2+4ts+4)}\\
&=2s\pm \sqrt{5s^2+4ts+4}
\end{align*}
Finally, we solve $z=x+x^{-1}$ to find the four solutions for $x$. In the case where
\[x+x^{-1}=2s+ \sqrt{5s^2+4ts+4},\]
we solve 
\[x^2-(2s+ \sqrt{5s^2+4ts+4})x+1=0.\]
Therefore, 
\begin{align*}
x&=\frac{2s+ \sqrt{5s^2+4ts+4}\pm\sqrt{\{2s+ \sqrt{5s^2+4ts+4})\}^2-4}}{2}\\
&=s+\frac{1}{2}\sqrt{5s^2+4ts+4}\pm\frac{1}{2}\sqrt{9s^2+4ts+4s\sqrt{5s^2+4ts+4}}
\end{align*}
The case where \[x+x^{-1}=2s- \sqrt{5s^2+4ts+4}\] is similar. Thus, we are done.
\end{proof}

\begin{lem}\label{Lem:distEigen2}
For any $d\in \NN$ and for $m_1,\cdots,m_{2d}\in \ZZ$, if $m_i>0$  for all $i\neq 2d$ and $m_{2d}\geq 0$, $\fC\fW(m_1,m_2,\cdots,m_{2d}) \fE$ has five distinct real eigenvalues (listed in \reflem{eigenValues2}), the absolute values of which are also distinct. Moreover, \[
s+\frac{1}{2}\sqrt{5s^2+4ts+4}+\frac{1}{2}\sqrt{9s^2+4ts+4s\sqrt{5s^2+4ts+4}},
\]
is the  eigenvalue that has the largest absolute value. Here, $t$ and $s$ are the trace and anti-trace of $W_1(m_1,m_2,\cdots,m_{2d})$, respectively.
\end{lem}
\begin{proof}
By \reflem{eigenValues2}, we have the formulas for the eigenvalues of $\fC\fW\fE$. By using \refprop{increasingEntry}, we can see that $s$ and $t$ are positive integers under the assumption that  $m_i>0$  for all $i\neq 2d$ and $m_{2d}\geq 0$. Therefore, the components $$s,\ \sqrt{5s^2+4ts+4},\ \text{and } \sqrt{9s^2+4ts+4s\sqrt{5s^2+4ts+4}}$$ in the formulas in \reflem{eigenValues2} are  positive real numbers bigger than or equal to $1$. Then, it follows immediately that \[
s+\frac{1}{2}\sqrt{5s^2+4ts+4}+\frac{1}{2}\sqrt{9s^2+4ts+4s\sqrt{5s^2+4ts+4}}.
\] is a positive real number bigger than $1$. Also, to see that the eigenvalues are real numbers, it is enough to show that $$\sqrt{9s^2+4ts-4s\sqrt{5s^2+4ts+4}}$$ is real, equivalently,
$$9s^2+4ts\geq 4s\sqrt{5s^2+4ts+4}.$$
Observe that 
\begin{align*}
    (9s^2+4ts)^2- (4s\sqrt{5s^2+4ts+4})^2&=81s^4+72ts^3+16t^2s^2-16s^2(5s^2+4ts+4)\\
&=81s^4+72ts^3+16t^2s^2-80s^4-64ts^3-64s^2\\
&=s^4+8ts^3+16s^2(t^2-4)\\
&>0.
\end{align*}
as $s>0$ and $t\geq 2$ by \reflem{trForm}.
As $9s^2+4ts$ and  $4s\sqrt{5s^2+4ts+4}$ are positive, we may conclude that 
\[9s^2+4ts > 4s\sqrt{5s^2+4ts+4}\]
and that 
\[9s^2+4ts- 4s\sqrt{5s^2+4ts+4}\] is positive.

Now, we want to show that the eigenvalues are distinct. To see this, we first compare the components, 
\begin{align*}
\sqrt{5s^2+4ts+4},\ \sqrt{9s^2+4ts-4s\sqrt{5s^2+4ts+4}}&,\  \text{and } \\
&\sqrt{9s^2+4ts+4s\sqrt{5s^2+4ts+4}}.
\end{align*}
Note that by the previous estimation, these are positive real numbers and that, obviously, 
$$\sqrt{9s^2+4ts+4s\sqrt{5s^2+4ts+4}}>\sqrt{9s^2+4ts-4s\sqrt{5s^2+4ts+4}}$$
as $s,t>0$.
Then, observe that 
\begin{align*}
9s^2+4ts+4s\sqrt{5s^2+4ts+4}-&(5s^2+4ts+4)\\
&=4(s^2-1)+4s\sqrt{5s^2+4ts+4}\\
&>0.
\end{align*}
as $s$ is a positive integer.
Also, observe that 
\begin{align*}
(5s^2+4ts+4)-&(9s^2+4ts-4s\sqrt{5s^2+4ts+4})\\
&=-4s^2+4+4s\sqrt{5s^2+4ts+4}\\
&=4\sqrt{5s^4+4ts^3+4s^2}-4\sqrt{s^4-2s^2+1}\\
&>0.
\end{align*}
as $t$ and $s$ are  positive integers. Therefore, we have that 
\begin{align*}
\sqrt{9s^2+4ts+4s\sqrt{5s^2+4ts+4}}&>\sqrt{5s^2+4ts+4}\\
&>\sqrt{9s^2+4ts-4s\sqrt{5s^2+4ts+4}}\\
&>0.
\end{align*}
Therefore, we can see that  \[
s+\frac{1}{2}\sqrt{5s^2+4ts+4}+\frac{1}{2}\sqrt{9s^2+4ts+4s\sqrt{5s^2+4ts+4}}.
\] is the unique eigenvalue that has the  largest absolute value.

Meanwhile, the third and fourth eigenvalues listed in \reflem{eigenValues2} are different from $1$. It follows from the facts that they are distinct and that, from the proof of \reflem{eigenValues2},  the product of them is equal to $1$. Hence, the absolute value of one of them lies between $1$ and the absolute value of the first eigenvalue and the absolute value of the other is less than $1$. Therefore, we also conclude that the second eigenvalue in \reflem{eigenValues2} is the unique eigenvalue that has the smallest absolute value. This follows from the facts that  the product of the first and second eigenvalues listed in \reflem{eigenValues2} is equal to $1$ and that the first eigenvalue has the largest absolute value. Thus, the five eigenvalues of $\fC \fW \fE$ are distinct real numbers, the absolute values of which are also distinct, and the positive eigenvalue \[
s+\frac{1}{2}\sqrt{5s^2+4ts+4}+\frac{1}{2}\sqrt{9s^2+4ts+4s\sqrt{5s^2+4ts+4}},
\] has the largest absolute value.
\end{proof}
\begin{rmk}\label{Rmk:PowEigen2}
Under the assumption in \reflem{distEigen2}, any positive $n^{th}$-power of $\fC \fW \fE$ has five distinct real eigenvalues which are the $n^{th}$-powers of the eigenvalues of $\fC \fW \fE$. The largest eigenvalue of $(\fC \fW \fE)^n$ is the $n^{th}$-power of the largest eigenvalue of $\fC \fW \fE$    
\end{rmk}

\subsection{The main theorem}\label{Sec:main}
Now, we prove the main theorem in the following form.

\begin{thm}\label{Thm:almostMinimal}
For each fully irreducible outer automorphism $\varphi$ in $\Out(F_2)$, there is a pseudo-Anosov element $\psi$ in $A^{-1}(\varphi)$ with $\lambda_\psi /\mu_\varphi<10$. 
\end{thm}
\begin{proof}
Recall the identifications in \refsec{OutF2} and \refsec{lifting}.
Let $\varphi$ be a fully irreducible outer automorphism in $\Out(F_2)$ and $M_\varphi$ the corresponding matrix in $\GL(\ZZ)$.
In \refsec{Penner},  we construct a pseudo-Anosov mapping class $M_\varphi^\star$ so that $A(M_\varphi^\star)=M_\varphi$. By \refsec{stdForm}, there is a matrix $N$ in $\GL(\ZZ)$ such that $NMN^{-1}$ is a standard form of $M$. Note that the standard form of $M$ is uniquely defined up to conjugation. As $A$ is surjective, there is a $\psi_N$ in $\cH_2$ such that $A(\psi_N)=N$. Then, $A(\psi_NM_\varphi^\star \psi_N^{-1})=NM_\varphi N^{-1}$.
The stretch factors are invariant under conjugation, that is, 
$$\lambda_{M_\varphi^\star}=\lambda_{\psi_NM_\varphi^\star \psi_N^{-1}} \text{ and }\mu_{M_\varphi}=\mu_{NM_\varphi N^{-1}}.$$
Therefore, without loss of generality, we may assume that $M_\varphi$ is of standard form, that is, $S(M)=M$.       
As in \refsec{hyperbolic} and \refsec{glideRef}, we consider four cases according to the isometry type of $M_\varphi$ in $\PSL{\ZZ}$ and the sign of $\tr(M_\varphi)$.

\textbf{Case 1} : $M_\varphi$ is hyperbolic and the sign of $\tr(M
_\varphi)$ is positive, that is, 
$$M_\varphi=W_1(m_1,m_2,\cdots, m_{2d})$$  
for some $d\in\NN$ and for some $m_1,m_2,\cdots,m_{2d}\in \NN$. We write $t=\tr(M_\varphi)$ and $s=\tr^*(M_\varphi)$.  The stretch factor $\mu_{M_\varphi}$ is the largest eigenvalue of $M_\varphi$,
$$\frac{t+\sqrt{t^2-4}}{2}$$
as remarked in \refsec{trAndStr}.
Then, 
$M_\varphi^\star$ is a pseudo-Anosov mapping class, the stretch factor of which is equal to the Perron-Frobenius eigenvalue of $\rho(w(M_\varphi))$. Note that  $\rho(w(M_\varphi))=\fC\fW(m_1,m_2,\cdots,m_{2d})$ by \refrmk{MRep}. By \reflem{distEigen}, the Perron-Frobenius eigenvalue is 
\[
\frac{t+4s+\sqrt{(t+4s)^2-4}}{2}.\]
Now, we have that 
\begin{align*}
\frac{\lambda_{M_\varphi^\star}}{\mu_{M_\varphi}}&=\frac{t+4s+\sqrt{(t+4s)^2-4}}{t+\sqrt{t^2-4}}\\
&=\frac{1+4s/t+\sqrt{(1+4s/t)^2-4/t^2}}{1+\sqrt{1-4/t^2}} \\
&<1+4s/t+\sqrt{(1+4s/t)^2-4/t^2}\\
&<2(1+4s/t)\\
&<2(1+4)\\
&=10
\end{align*}
The first and second inequalities are implied by the fact that $t$ and $s$ are integers bigger than or equal to $2$ by \refprop{increasingEntry}. The third inequality follows from \reflem{comparisionEven}. This shows the first case.

\textbf{Case 2} : 
$M_\varphi$ is hyperbolic and the sign of $\tr(M
_\varphi)$ is negative, that is, 
$$M_\varphi=-W_1(m_1,m_2,\cdots, m_{2d})$$  
for some $d\in\NN$ and for some $m_1,m_2,\cdots,m_{2d}\in \NN$. We write $t=\tr(W_1)$ and $s=\tr^*(W_1)$. In this case, the stretch factor $\mu_{M_\varphi}$ is the largest eigenvalue of $W_1$,
$$\frac{t+\sqrt{t^2-4}}{2}$$
as remarked in \refsec{trAndStr}.

On the other hand, 
$M_\varphi^\star$ constructed in \refsec{Penner} is a pseudo-Anosov mapping  class. Now, we want to estimate the stretch factor of $M_\varphi^\star$. To see this, we consider the following equality:
$$M_\varphi^\star\circ M_\varphi^\star=\zeta(w(M))\circ \zeta(w(M))=\zeta(w(M)w(M)).$$
By \refrmk{MRep}, 
$$\rho(w(M)w(M))=\fC \fW \fC\fW$$
where $\fW=\fW(m_1,m_2,\cdots,m_{2d}).$ By \refcor{PennCstr}, the stretch factor of $M_\varphi^\star\circ M_\varphi^\star$ is the Perron-Frobenius eigenvalue of $(\fC \fW)^2$. Since $$\lambda_{M_\varphi^\star\circ M_\varphi^\star}=\lambda_{M_\varphi^\star}^2$$, by \refrmk{PowEigen}, we can see that $\lambda_{M_\varphi^\star}$ is equal to the largest eigenvalue of $\fC\fW$. Therefore, 
by \reflem{distEigen}, 
$$\lambda_{M_\varphi^\star}=
\frac{t+4s+\sqrt{(t+4s)^2-4}}{2}.$$
Then, by the computation in Case $1$, we can also obtain that 
$$\frac{\lambda_{M_\varphi^\star}}{\mu_{M_\varphi}}<10.$$

\textbf{Case 3} :
$M_\varphi$ is a glide reflection and the sign of $\tr(M
_\varphi)$ is positive, that is, 
$$M_\varphi=W_1(m_1,m_2,\cdots,m_{2d-2}, m_{2d-1})\tilde{e}$$  
for some $d\in\NN$ and for some $m_1,m_2,\cdots,m_{2d-1}\in \NN$. Here, $\tilde{e}$ is defined in \refsec{lifting}. We write $t=\tr(W_1)$ and $s=\tr^*(W_1)$. Then, $s=\tr(W_1\tilde{e})$ and $t=\tr^*(W_1\tilde{e})$.
In this case, as remarked in \refsec{trAndStr}, the stretch factor $\mu_{M_\varphi}$ is the largest eigenvalue of $M_\varphi$ and the equality 
$$\mu_{M_\varphi}-\mu_{M_\varphi}^{-1}=\tr(W_1\tilde{e})=s$$ holds.

On the other hand, 
$M_\varphi^\star$ constructed in \refsec{Penner} is a pseudo-Anosov mapping  class. Now, we want to estimate the stretch factor of $M_\varphi^\star$. To see this, we consider the following equality:
$$M_\varphi^\star\circ M_\varphi^\star=\zeta(w)$$
where 
\begin{align*}
w=\delta_5 [\prod_{k=1}^{d-1}
(\delta_1\delta_3)^{m_{2k-1}}&(\delta_2\delta_4)^{m_{2k}}](\delta_1\delta_3)^{m_{2d-1}}\\
&\times\delta_5[\prod_{k=1}^{d-1}(\delta_2\delta_4)^{m_{2k-1}}(\delta_1\delta_3)^{m_{2k}}](\delta_2\delta_4)^{m_{2d-1}}
\end{align*} if $d>1$ and 
$$w=\delta_5(\delta_1\delta_3)^{m_{2d-1}}\delta_5(\delta_2\delta_4)^{m_{2d-1}}$$ if $d=1$.
By \refthm{MRep}, \refrmk{MRep}, and \reflem{eigenValues2}, 
$$\rho(w)=(\fC\fW)(\fC\fE\fW\fE)=(\fC\fW)(\fE\fC\fW\fE)=(\fC\fW\fE)^2$$
where $\fW=\fW(m_1,m_2,\cdots, m_{2d-1})$. Then, by \refcor{PennCstr}, the stretch factor of $M_\varphi^\star\circ M_\varphi^\star$ is the Perron-Frobenius eigenvalue of $(\fC\fW\fE)^2$. 
Since $$\lambda_{M_\varphi^\star\circ M_\varphi^\star}=\lambda_{M_\varphi^\star}^2$$, by \refrmk{PowEigen2}, we can see that $\lambda_{M_\varphi^\star}$ is equal to the largest eigenvalue of $\fC\fW\fE$.
Therefore, 
by \reflem{distEigen2}, 
$$\lambda_{M_\varphi^\star}=
s+\frac{1}{2}\sqrt{5s^2+4ts+4}+\frac{1}{2}\sqrt{9s^2+4ts+4s\sqrt{5s^2+4ts+4}}
.$$
Now, we have that 
\begin{align*}
\frac{\lambda_{M_\varphi^\star}}{\mu_{M_\varphi}}&< \frac{\lambda_{M_\varphi^\star}}{\mu_{M_\varphi}-\mu_{M_\varphi}^{-1}}\\
&= s^{-1}\left[ s+\frac{1}{2}\sqrt{5s^2+4ts+4}+\frac{1}{2}\sqrt{9s^2+4ts+4s\sqrt{5s^2+4ts+4}}\right]\\
&=1+\frac{1}{2}\sqrt{5+4\frac{t}{s}+\frac{4}{s^2}}+\frac{1}{2}\sqrt{9+4\frac{t}{s}+4\sqrt{5+4\frac{t}{s}+\frac{4}{s^2}}}=(*).
\end{align*}
The first inequality follows from the fact that $0<\mu_{M_\varphi}-\mu_{M_\varphi}^{-1}<\mu_{M_\varphi}$.

To estimate an upper bound for the equation $(*)$, we need to consider two cases according to whether $d=1$ or not.

If $d>1$, then
\begin{align*}
(*)&<1+\frac{1}{2}\sqrt{5+4+1}+\frac{1}{2}\sqrt{9+4+4\sqrt{5+4+1}} \\
&<1+2+3\\
&=6.
\end{align*}
The first inequality follows from \reflem{comparisionOdd} and the fact that $s\geq 2$ by \refprop{increasingEntry}. By direct computation, we can see the second inequality.

Finally, we consider the case of $d=1$. As
$W_1(m_1)=A_1^{m_1}=A_{m_1}$, $t=2$ and $s=m_1$. Then,  $$\frac{t}{s}=\frac{2}{m_1}\leq 2 \text{ and } s=m_1\geq 1.$$
Therefore, we have 
\begin{align*}
(*)&<1+\frac{1}{2}\sqrt{5+8+4}+\frac{1}{2}\sqrt{9+8+4\sqrt{5+8+4}} \\
&<1+3+3\\
&=7.
\end{align*}
Thus, 
$$\frac{\lambda_{M_\varphi^\star}}{\mu_{M_\varphi}}<10.$$

\textbf{Case 4} : $M_\varphi$ is a glide reflection and the sign of $\tr(M
_\varphi)$ is negative, that is, 
$$M_\varphi=-W_1(m_1,m_2,\cdots,m_{2d-2}, m_{2d-1})\tilde{e}$$  
for some $d\in\NN$ and for some $m_1,m_2,\cdots,m_{2d-1}\in \NN$.  We write $t=\tr(W_1)$ and $s=\tr^*(W_1)$. Then, $s=\tr(W_1\tilde{e})$ and $t=\tr^*(W_1\tilde{e})$.
As discussed in \refsec{trAndStr}, the stretch factor $\mu_{M_\varphi}$ is equal to the largest eigenvalue of $W_1\tilde{e}$ and the equality 
$$\mu_{M_\varphi}-\mu_{M_\varphi}^{-1}=\tr(W_1\tilde{e})=s$$ holds.

On the other hand, 
$M_\varphi^\star$ constructed in \refsec{Penner} is a pseudo-Anosov mapping  class. Now, we want to estimate the stretch factor of $M_\varphi^\star$. To see this, we consider the following equality:
$$M_\varphi^\star\circ M_\varphi^\star=\zeta(w)$$
where 
\begin{align*}
w=\delta_5 [\prod_{k=1}^{d-1}
(\delta_1\delta_3)^{m_{2k-1}}&(\delta_2\delta_4)^{m_{2k}}](\delta_1\delta_3)^{m_{2d-1}}\\
&\times\delta_5[\prod_{k=1}^{d-1}(\delta_2\delta_4)^{m_{2k-1}}(\delta_1\delta_3)^{m_{2k}}](\delta_2\delta_4)^{m_{2d-1}}
\end{align*} if $d>1$ and 
$$w=\delta_5(\delta_1\delta_3)^{m_{2d-1}}\delta_5(\delta_2\delta_4)^{m_{2d-1}}$$ if $d=1$.
By \refthm{MRep}, \refrmk{MRep}, and \reflem{eigenValues2}, 
$$\rho(w)=(\fC\fW)(\fC\fE\fW\fE)=(\fC\fW)(\fE\fC\fW\fE)=(\fC\fW\fE)^2$$
where $\fW=\fW(m_1,m_2,\cdots, m_{2d-1})$. Then, by \refcor{PennCstr}, the stretch factor of $M_\varphi^\star\circ M_\varphi^\star$ is the Perron-Frobenius eigenvalue of $(\fC\fW\fE)^2$. 
Since $$\lambda_{M_\varphi^\star\circ M_\varphi^\star}=\lambda_{M_\varphi^\star}^2$$, by \refrmk{PowEigen2}, we can see that $\lambda_{M_\varphi^\star}$ is equal to the largest eigenvalue of $\fC\fW\fE$.
Therefore, 
by \reflem{distEigen2}, 
$$\lambda_{M_\varphi^\star}=
s+\frac{1}{2}\sqrt{5s^2+4ts+4}+\frac{1}{2}\sqrt{9s^2+4ts+4s\sqrt{5s^2+4ts+4}}
.$$ 
Thus, by the  computation in Case $3$, we can conclude that 
$$\frac{\lambda_{M_\varphi^\star}}{\mu_{M_\varphi}}<10.$$
\end{proof}

\section{Application : The geodesic counting problem} \label{Sec:GeodesicCounting}

The geodesic counting problem is one of the famous topic in dynamics. Hence, there are many variation of geodesic counting problem. Espceically, we want to emphasize the geodesic counting problems \cite{EskinMirzakhani11} in a Moduli space  and \cite{KapovichPfaff18} in the outer automorphism groups of a free group.

In this section, we discuss the geodesic counting problem in handlebody groups, inspired by \cite{EskinMirzakhani11},\cite{ChowlaCowlesCowles80} and \cite{KapovichPfaff18}. Let $S_{g,p}$ be the hyperbolic surface with genus $g$ and $p$ punctures.
For $R>0$, let $N_{g,p}
(R)$ be the number of conjugacy classes of pseudo-Anosov elements of $\Mod(S_{g,p})$ of translation length at most $R$ in the Teichm\"uller space $\cT(S_{g,p})$.

\begin{thm}[ \cite{EskinMirzakhani11}]\label{Thm:countInMod}
The number $N_{g,p}(R)$ is asymptotic to $e^{hR}/hR$ as $R\to \infty$, where $h=6g-6+2p$. Namely, 
$$\frac{N_{g,p}(R)}{e^{hR}/hR} \to 1 \text{ as }R\to \infty.$$
\end{thm}

For the once punctured torus $S_{1,1}$, we have that the number
$N_{1,1}(R)$ is asymptotic to $e^{2R}/2R$ as $R \to \infty$. Note that $\SL(\ZZ)$ is the (orientation preserving) mapping class group of the once punctured torus. 

For $R>0$ and for $n\in \NN$, we  denote the number of conjugacy classes of fully irreducible outer automoprhisms $\varphi$ of $\Out(F_n)$ with $\log \mu_\varphi \leq R$ by $\mathfrak{N}_n(R)$, following \cite{KapovichPfaff18}. Also, for $R>0$ and for $g\in \NN$,  $\mathfrak{H}_g(R)$ denotes the number of conjugacy classes of pseudo-Anosov elements $\psi$ in $\cH_g$ of translation length at most $R$ in $\cT(\partial V_g)$, namely, $\log(\psi)\leq R$. Note that by \refthm{discretSpecH}, $\fkH_g(R)$ is finite for  all $R>0$.

\begin{prop}\label{Prop:countInTorus}
$N_{1,1}(R)/2\leq \fkN_2(R)$ for all $R>0$.
\end{prop}
\begin{proof}
Recall that  $\Out(F_2)=\Mod^\pm(S_{1,1})=\GL(\ZZ)$  and that the (orientation preserving) mapping class group $\Mod(S_{1,1})$ of the once punctured torus is $\SL(\ZZ)$. Note that 
$\GL(\ZZ)= \SL(\ZZ) \cup \tilde{e} \cdot\SL(\ZZ)$.
Here, $$\tilde{e}=\tilde{e}^{-1}=\begin{pmatrix}
    0&1\\
    1&0
\end{pmatrix}.$$
Let $M$ be a hyperbolic matrix in $\SL(\ZZ)$. Then,
$$\cO_{\GL(\ZZ)}(M)= \cO_{\SL(\ZZ)}(M) \cup \cO_{\SL(\ZZ)}(\tilde{e}M\tilde{e})$$ 
As $\tr(M)=\tr(\tilde{e}M\tilde{e})$ and $1=\det(M)=\det(\tilde{e}M\tilde{e})$, we can see that $\cO_{\SL(\ZZ)}(M)$ and  $\cO_{\SL(\ZZ)}(\tilde{e}M\tilde{e})$ are conjugacy classes of fully irreducible outer automoprhisms having the same stretch factor $\mu_M$. Therefore, this implies that $$N_{1,1}(R) \leq 2 \fkN_2(R)$$ for all $R>0$, even though $\cO_{\SL(\ZZ)}(M)$ and  $\cO_{\SL(\ZZ)}(\tilde{e}M\tilde{e})$
 may coincide. 
\end{proof}

\begin{thm}\label{Thm:countH2}
For any $\epsilon>0$, there is a real number $R_\epsilon>0$ such that 
$$(1-\epsilon)\frac{e^{2R}}{400R}< \fkH_2(R)$$ for all $R> R_\epsilon$.
\end{thm}
\begin{proof}  
First, observe that for any $f\in \cH_2$, $$A(\cO_{\cH_2}(f))\subset \cO_{\Out(F_2)}(A(f)).$$ Hence, for any $g,h\in \cH_2$ with $\cO_{\Out(F_2)}(A(g))\neq \cO_{\Out(F_2)}(A(h))$, $\cO_{\cH_2}(g)\neq \cO_{\cH_2}(h)$

Choose $s>\log 10$ and assume that an outer automoprhism $\varphi$ in $\Out(F_2)$ is a fully irreducible with $\log \mu_\varphi < s-\log10.$ 
By \refthm{almostMinimal}, we can take a pseudo-Anosov $\psi$ in $A^{-1}(\varphi)$ so that $\log \lambda_\psi<s$. Then, we have 
$$ \fkN_2(s-\log 10)\leq \fkH_2(s)$$
and by \refprop{countInTorus},
$$ \frac{N_{1,1}(s-\log 10)}{2}\leq \fkH_2(s).$$

Now, we choose $\epsilon>0$. By \refthm{countInMod}, there is an real number $r_\epsilon>0$ such that  
$$1-\epsilon < \frac{N_{1,1}(r)}{e^{2r}/2r}< 1+\epsilon$$
for all $r>r_\epsilon$. Set $R_\epsilon=r_\epsilon+\log 10$
 and choose $R>R_\epsilon
$. Then, 
\begin{align*}
    \fkH_2(R)&\geq \frac{N_{1,1}(R-\log 10)}{2}\\
    &> \frac{1}{2}\cdot(1-\epsilon)\frac{e^{2(R-\log 10)}}{2(R-\log 10)}\\
    &=(1-\epsilon)\frac{e^{2R}}{400(R-\log 10)}\\
    &>(1-\epsilon)\frac{e^{2R}}{400R} \ .    
\end{align*}
Thus, for each $\epsilon>0$, there is a real number $R_\epsilon>0$ such that 
$$(1-\epsilon)\frac{e^{2R}}{400R}<\fkH_2(R)$$
for all $R>R_\epsilon$.
\end{proof}

\bibliographystyle{alpha}
\bibliography{biblio.bib}

\begin{thebibliography}{BKKS21}

\bibitem[BBM07]{BestvinaBuxMargalit07}
Mladen Bestvina, Kai-Uwe Bux, and Dan Margalit.
\newblock Dimension of the {T}orelli group for {${\rm Out}(F_n)$}.
\newblock {\em Invent. Math.}, 170(1):1--32, 2007.

\bibitem[BH92]{BestvinaHandel92}
Mladen Bestvina and Michael Handel.
\newblock Train tracks and automorphisms of free groups.
\newblock {\em Ann. of Math. (2)}, 135(1):1--51, 1992.

\bibitem[BKKS21]{BaikKimKwakShin21}
Hyungryul Baik, Changsub Kim, Sanghoon Kwak, and Hyunshik Shin.
\newblock On translation lengths of {A}nosov maps on the curve graph of the
  torus.
\newblock {\em Geom. Dedicata}, 214:399--426, 2021.

\bibitem[CCC80]{ChowlaCowlesCowles80}
S.~Chowla, J.~Cowles, and M.~Cowles.
\newblock On the number of conjugacy classes in {${\rm SL}(2,\,{\bf Z})$}.
\newblock {\em J. Number Theory}, 12(3):372--377, 1980.

\bibitem[Che22]{Chesser22}
Marissa Chesser.
\newblock Stable subgroups of the genus 2 handlebody group.
\newblock {\em Algebr. Geom. Topol.}, 22(2):919--971, 2022.

\bibitem[CL23]{ChesserLeininger23}
Marissa Chesser and Christopher~J Leininger.
\newblock Purely pseudo-anosov subgroups of the genus two handlebody group.
\newblock {\em arXiv preprint arXiv:2304.02570}, 2023.

\bibitem[EM11]{EskinMirzakhani11}
Alex Eskin and Maryam Mirzakhani.
\newblock Counting closed geodesics in moduli space.
\newblock {\em J. Mod. Dyn.}, 5(1):71--105, 2011.

\bibitem[FH82]{FloydHatcher82}
W.~Floyd and A.~Hatcher.
\newblock Incompressible surfaces in punctured-torus bundles.
\newblock {\em Topology Appl.}, 13(3):263--282, 1982.

\bibitem[FM12]{FarbMargalit12}
Benson Farb and Dan Margalit.
\newblock {\em A primer on mapping class groups}, volume~49 of {\em Princeton
  Mathematical Series}.
\newblock Princeton University Press, Princeton, NJ, 2012.

\bibitem[Gui00]{Guirardel00}
Vincent Guirardel.
\newblock Dynamics of {${\rm Out}(F_n)$} on the boundary of outer space.
\newblock {\em Ann. Sci. \'{E}cole Norm. Sup. (4)}, 33(4):433--465, 2000.

\bibitem[Hai08]{Hain08}
Richard Hain.
\newblock Relative weight filtrations on completions of mapping class groups.
\newblock In {\em Groups of diffeomorphisms}, volume~52 of {\em Adv. Stud. Pure
  Math.}, pages 309--368. Math. Soc. Japan, Tokyo, 2008.

\bibitem[Hen20]{Hensel20}
Sebastian Hensel.
\newblock A primer on handlebody groups.
\newblock In {\em Handbook of group actions. {V}}, volume~48 of {\em Adv. Lect.
  Math. (ALM)}, pages 143--177. Int. Press, Somerville, MA, [2020] \copyright
  2020.

\bibitem[HH21]{HamenstadtHensel21}
Ursula Hamenst\"{a}dt and Sebastian Hensel.
\newblock The geometry of the handlebody groups {II}: {D}ehn functions.
\newblock {\em Michigan Math. J.}, 70(1):23--53, 2021.

\bibitem[KP18]{KapovichPfaff18}
Ilya Kapovich and Catherine Pfaff.
\newblock Counting conjugacy classes of fully irreducibles: double exponential
  growth.
\newblock {\em arXiv preprint arXiv:1801.07471}, 2018.

\bibitem[Luf78]{Luft78}
E.~Luft.
\newblock Actions of the homeotopy group of an orientable {$3$}-dimensional
  handlebody.
\newblock {\em Math. Ann.}, 234(3):279--292, 1978.

\bibitem[Mas86]{Masur86}
Howard Masur.
\newblock Measured foliations and handlebodies.
\newblock {\em Ergodic Theory Dynam. Systems}, 6(1):99--116, 1986.

\bibitem[McC85]{McCullough85}
Darryl McCullough.
\newblock Twist groups of compact {$3$}-manifolds.
\newblock {\em Topology}, 24(4):461--474, 1985.

\bibitem[McM63]{McMillan63}
D.~R. McMillan, Jr.
\newblock Homeomorphisms on a solid torus.
\newblock {\em Proc. Amer. Math. Soc.}, 14:386--390, 1963.

\bibitem[Oer02]{Oertel02}
Ulrich Oertel.
\newblock Automorphisms of three-dimensional handlebodies.
\newblock {\em Topology}, 41(2):363--410, 2002.

\bibitem[Pen88]{Penner88}
Robert~C. Penner.
\newblock A construction of pseudo-{A}nosov homeomorphisms.
\newblock {\em Trans. Amer. Math. Soc.}, 310(1):179--197, 1988.

\bibitem[PS16]{PapaSu16}
Athanase Papadopoulos and Weixu Su.
\newblock Thurston's metric on {T}eichm\"{u}ller space and the translation
  distances of mapping classes.
\newblock {\em Ann. Acad. Sci. Fenn. Math.}, 41(2):867--879, 2016.

\bibitem[Sco83]{Scott83}
Peter Scott.
\newblock The geometries of {$3$}-manifolds.
\newblock {\em Bull. London Math. Soc.}, 15(5):401--487, 1983.

\bibitem[SS15]{ShinStrenner15}
Hyunshik Shin and Bal\'{a}zs Strenner.
\newblock Pseudo-{A}nosov mapping classes not arising from {P}enner's
  construction.
\newblock {\em Geom. Topol.}, 19(6):3645--3656, 2015.

\bibitem[Thu88]{Thurston88}
William~P. Thurston.
\newblock On the geometry and dynamics of diffeomorphisms of surfaces.
\newblock {\em Bull. Amer. Math. Soc. (N.S.)}, 19(2):417--431, 1988.

\bibitem[Zie62]{Zieschang61}
Heiner Zieschang.
\newblock \"{U}ber einfache {K}urven auf {V}ollbrezeln.
\newblock {\em Abh. Math. Sem. Univ. Hamburg}, 25:231--250, 1961/62.

\end{thebibliography}

\end{document}